\documentclass[reqno]{amsart}
\usepackage{amssymb}
\usepackage{amsmath}
\usepackage{amsthm}
\usepackage{mathtools}
\usepackage{placeins}
\usepackage{color}

\usepackage{hyperref}

\hypersetup{
    unicode=false,          
    pdftoolbar=true,        
    pdfmenubar=true,        
    pdffitwindow=false,     
    pdfstartview={FitH},    
    pdftitle={Multi-speed solitary waves of nonlinear schr\"odinger
  systems: theoretical and numerical analysis},   
    pdfauthor={Fanny Delebecque <fanny.delebecque@math.univ-toulouse.fr>, Stefan Le Coz
    <slecoz@math.univ-toulouse.fr>, Rada Maria Weishaeupl <rada.weishaeupl@univie.ac.at>},     
    pdfsubject={Multi-speed solitary waves of nonlinear schr\"odinger
  systems: theoretical and numerical analysis},   
    pdfcreator={Creator},   
    pdfproducer={Producer}, 
    pdfkeywords={solitary waves}{solitons}{ nonlinear Schrodinger systems} , 
    pdfnewwindow=true,      
    colorlinks=false,       
    linkcolor=red,          
    citecolor=green,        
    filecolor=magenta,      
    urlcolor=cyan,           
}

\newtheorem{theorem}{Theorem}
\newtheorem{proposition}{Proposition}
\newtheorem{lemma}[proposition]{Lemma}

\theoremstyle{remark}
\newtheorem{remark}[proposition]{Remark}

\DeclarePairedDelimiter{\norm}{\lVert}{\rVert}
\DeclarePairedDelimiter{\abs}{\lvert}{\rvert}

\newcommand{\psld}[2]{\left( #1,#2 \right)_{2}}

\newcommand{\dual}[2]{\left\langle #1,#2 \right\rangle}

\newcommand{\eps}{\varepsilon}
\newcommand{\om}{\omega}
\newcommand{\tildo}{\tilde{\omega}}
\newcommand{\tildg}{\tilde{\gamma}}
\newcommand{\tildx}{\tilde{x}}
\newcommand{\tildR}{\tilde{R}}

\newcommand{\N}{\mathbb{N}}
\newcommand{\R}{\mathbb{R}}

\newcommand{\calC}{\mathcal{C}}
\newcommand{\calO}{\mathcal{O}}

\newcommand{\calE}{\mathcal{E}}
\newcommand{\calP}{\mathcal{P}}
\newcommand{\calN}{\mathcal{N}}
\newcommand{\calPlocu}{\mathcal{P}_{\textrm{loc}}^{1}}
\newcommand{\calPlocd}{\mathcal{P}_{\textrm{loc}}^{2}}
\newcommand{\calPlocj}{\mathcal{P}_{\textrm{loc}}^{j}}

\newcommand{\hu}{H^1(\R)}
\newcommand{\hhu}{H^1\times H^1}

\renewcommand{\leq}{\leqslant}
\renewcommand{\geq}{\geqslant}

\DeclareMathAlphabet{\mathpzc}{OT1}{pzc}{m}{it}
\renewcommand{\Re}{\mathcal R\!\mathpzc{e}}
\renewcommand{\Im}{\mathcal I\!\mathpzc{m}}

\begin{document}

\title[Multi-speed solitary waves of nonlinear Schr\"odinger
systems]{Multi-speed solitary waves\\ of nonlinear Schr\"odinger
  systems: \\theoretical and numerical analysis}

\author[F.~Delebecque]{Fanny Delebecque}
\thanks{The work of F. D.  is 
  partially supported by PHC AMADEUS
  2014 31471ZK}

\author[S.~Le Coz]{Stefan Le Coz}
\thanks{The work of S. L. C. is 
  partially supported by ANR-11-LABX-0040-CIMI within the
  program ANR-11-IDEX-0002-02,  ANR-14-CE25-0009-01 and PHC AMADEUS
  2014 31471ZK}

\address[Fanny Delebecque and Stefan Le Coz]{Institut de Math\'ematiques de Toulouse,
  Universit\'e Paul Sabatier,
  118 route de Narbonne, 31062 Toulouse Cedex 9,
  France}
\email[Fanny Delebecque]{fanny.delebecque@math.univ-toulouse.fr}
\email[Stefan Le Coz]{slecoz@math.univ-toulouse.fr}

\author[R. M.~Weish\"aupl]{Rada M. Weish\"aupl}

\address[Rada M. Weish\"aupl]{
  Faculty of Mathematics, University of Vienna, 
  Oskar-Morgenstern-Platz 1, 
  1090 Vienna, Austria}
\email[Rada M. Weish\"aupl]{rada.weishaeupl@univie.ac.at}

\thanks{
  The work of R.M.W. is supported by the  FWF
  Hertha-Firnberg Program, Grant T402-N13 and  Austrian-French Project
WTZ-Amad\'ee FR 18/2014}

\date{\today}

\begin{abstract}
  We consider a system of coupled nonlinear Schr\"odinger equations in one space dimension. First, we prove the existence of multi-speed solitary waves, {\it i.e} solutions to the system with each component behaving at large times as a solitary wave. Then, we investigate numerically the interaction of two solitary waves supported each on one component. Among the possible outcomes, we find elastic and inelastic interactions, collision with mass extraction and reflexion. 
\end{abstract}

\maketitle

\tableofcontents

\section{Introduction}

We consider the following nonlinear Schr\"odinger system:
\begin{equation}\label{eq:nls}\tag{NLS}
  \left
    \{
    \begin{aligned}
      i\partial_t u_1+\partial_{xx} u_1+\mu_1|u_1|^2u_1+\beta|u_2|^2u_1&=0,\\
      i\partial_t u_2+\partial_{xx} u_2+\mu_2|u_2|^2u_2+\beta|u_1|^2u_2&=0,
    \end{aligned}
  \right.
\end{equation}
where  for $j=1,2$ we have $u_j:\mathbb R\times\mathbb
R\rightarrow\mathbb C$, $\mu_j>0$, and $\beta\in\R\setminus\{0\}$. 

When $\mu_{1}=\mu_{2}=\beta$, system~\eqref{eq:nls}, also called \emph{Manakov system} has been introduced by Manakov (see~\cite{Ma74} for example) as an asymptotic model for the propagation of electric fields in waveguides. In this particular case, it is to be noticed that the usual roles of $x$ and $t$ are inverted to study the evolution of the electrical field along the propagation axis.

It has also been used later on to model the evolution of light in optical fiber links. One of the main limiting effects of transmission in optical fiber links is due to the polarization mode dispersion (PMD). It can be explained by the \emph{birefringence effect}, i.e the fact that the electric field is a vector field and that the refraction index of the medium depends on the polarization state (see e.g~\cite{AbPrTr04,Ab07}). The evolution of two polarized modes of an electrical field in a birefringent optical fiber link can indeed be modeled by~\eqref{eq:nls} in the case where $\mu_{1}=\mu_{2}$ and $\beta$ measures the strength of the cross phase modulation which depends of the fiber (see~\cite{Ma74}). Randomly varying birefringence is studied adding random coefficients in both nonlinearity and coupling terms of~\eqref{eq:nls} (see for example~\cite{GAMA06})

In higher dimensions, systems of nonlinear coupled Schr\"odinger equations appears in various physical situations such as the modeling of the interaction of two Bose-Einstein condensates 
in different spin states.

Systems of type~\eqref{eq:nls} have also been studied from the mathematical
point of view. When $\mu_1=\mu_2=\beta$, in dimension 1, the system~\eqref{eq:nls} has the particularity to be completely integrable. Hence explicit calculations of solutions are
possible and one can exhibit a variety of ``truly'' nonlinear 
solutions like solitons, multi-solitons or breathers (see e.g. the book
\cite{AbPrTr04}). The integrability property is however not robust, and the slightest change in the parameters $\mu_1$, $\mu_2$ and
$\beta$ destroys it. Many works (see, among many others,
\cite{AmCo07,BaWa06,deLo08,Si07}) have been devoted to the study of the stationary
version of~\eqref{eq:nls}
\begin{equation}
  \label{eq:snls}
  \left
    \{
    \begin{aligned}
      \partial_{xx} \phi_1+\mu_1|\phi_1|^2\phi_1+\beta|\phi_2|^2\phi_1=\omega_1\phi_1,\\
      \partial_{xx} \phi_2+\mu_2|\phi_2|^2\phi_2+\beta|\phi_1|^2\phi_2=\omega_2\phi_2,
    \end{aligned}
  \right.
\end{equation}
that one obtains when looking for standing
waves solutions 
\[
(u_1,u_2)(t,x)=(e^{i\omega_1 t}\phi_1(x),e^{i\omega_2
  t}\phi_2(x)).
\]
When standing waves exist, it is natural to study their stability and
again many works have been devoted to this problem (see, again among
many others,~\cite{Ha12,MaMoPe10,MoPeSq10,Oh96}). 
Existence and stability of standing waves are often proved using
variational techniques. The analysis of \eqref{eq:snls} through
variational techniques is very subtle and the introduction of 
new ideas is necessary to understand the full picture (see \cite{NgWa15}). 
Note that
\eqref{eq:nls} is Galilean-invariant. Hence a Galilean
transform modifies a standing wave into a solitary wave traveling at
some non-zero speed. 

Our goal in this paper is to provide a new point of view on the study
of this system. We aim at understanding better the behavior in large times of solutions
starting at initial time as two scalar solitary waves carried by the
two different components. We will use a mixture of theoretical and
numerical tools, a combination seldom seen when dealing with this kind
of problems.

First, we propose to push further a study initiated in
\cite{IaLe14} on the multi-speed solitary waves of system~\eqref{eq:nls} and followed up for a different nonlinearity in~\cite{WaCu15}. A
multi-speed solitary wave is a solution of~\eqref{eq:nls} which behaves at
large time as two solitary waves. Here and as in~\cite{IaLe14}, we restrict ourselves to the
case where the composing solitary waves are each carried on only one component of the
system. In other words, taken independently, each component behaves as a
scalar solitary wave at large time. Our first aim is to remove the high speed
assumption under which the main result in~\cite{IaLe14} was proved. We therefore consider the system~\eqref{eq:nls} in dimension $1$ and benefit from the fact that scalar solitary waves are in that case orbitally stable (see e.g~\cite{CaLi82}). 

Our next aim is to investigate further the properties of
multi-speed solitary waves solutions when they are crossing
at positive time. Our theoretical result (Theorem~\ref{theorem}) indeed
only guarantees existence of multi-speed solitary wave solutions to~\eqref{eq:nls} when no
interaction can occur at large time between the composing waves. 
But what happens when two solitary waves carried by different components collide?
Due to the possible complexity of the phenomenon and the lack of
appropriate  theoretical tools to study it, we proceed the following
numerical experiment.  We take as initial data  solitons on each components, both
away from $0$ but facing each other for the direction of
propagation. Among the possible outcomes, we find elastic and inelastic interactions, interaction with mass exctraction, and reflexion.

\subsection{The theoretical result}

Before stating our main theoretical result, let us give a few preliminaries. 

Let $Q_{\omega}\in\hu$ be the unique positive radial ground state  solution to 
\begin{equation}
  \label{eq:Q-omega}
  -\partial_{xx} Q_{\omega}+\omega Q_{\omega}-|Q_{\omega}|^2Q_{\omega}=0,\qquad Q_{\omega}>0,\quad Q_{\omega} \in H^1_{\mathrm{rad}}(\R).
\end{equation}
\noindent
From simple calculations we note that the following scaling occurs: 
\begin{equation}\label{eq:snls-basic-intro}
  Q_{\omega}(x):=\sqrt{\omega}\ Q_{1}(\sqrt{\omega} x).
\end{equation}
For $j=1,2$, consider $\omega_j>0,\gamma_j\in\R$, $x_j,v_j\in \R$ and define  
\begin{equation}
  \label{eq:soliton-j}
  R_j(t,x)
  =  e^{i(\omega_j t-\frac{1}{4}|v_j|^2t+\frac{1}{2}v_j\cdot
    x+\gamma_j)}\sqrt{\frac{1}{\mu_j}} Q_{
    \omega_{j}}(x-v_jt-x_j).
\end{equation}
The function $R_{j}$ is   a solitary wave solution to
\begin{equation}
  \label{eq:nls-j}
  i\partial_{t}u+\partial_{xx}u+\mu_{j}\abs{u}^2u=0.
\end{equation}  

In this paper, we want to investigate the existence of solutions to
\eqref{eq:nls} where each component behaves like a solitary wave $R_j$
solution to the scalar equation~\eqref{eq:nls-j}. Our main
theoretical result is the following. 

\begin{theorem}
  \label{theorem}
  Let $\mu_{1},\mu_{2}>0$ and $\beta\in \mathbb{R}\setminus\{0\}$. For
  $j=1,2$, take $v_{j}, x_{j}, \gamma_{j}\in\mathbb{R}$,
  $\omega_{j}>0$ and
  consider the ground state profile  $Q_{\omega_j}$ solution to~\eqref{eq:Q-omega} and the soliton $R_j$ defined  in
~\eqref{eq:soliton-j}. 
  Then, there exist $C>0$, $T_{0}>0$ and
  $\begin{pmatrix}u_1,u_2\end{pmatrix}$ solution to~\eqref{eq:nls} on
  the time interval $[T_{0},+\infty)$ such that  for all $t\in
  [T_{0},+\infty)$, we have
  \begin{equation*}
    \norm*{\begin{pmatrix}u_1(t)\\u_2(t)\end{pmatrix}-\begin{pmatrix}R_1(t)\\R_2(t)\end{pmatrix}
    }_{H^1\times H^1}\leq C e^{-\sqrt{\omega_{*}}v_{*}t},
  \end{equation*}
  where $\omega_{*}=\frac{1}{2304}\min\{\omega_{1},\omega_{2}\}$ and $v_{*}=\abs{v_{1}-v_{2}}$.
\end{theorem}

\begin{remark}
  Compare to~\cite[Theorem 1]{IaLe14}, the main differences are the
  following. Our result is valid for any speeds, whereas the one in
~\cite{IaLe14} required a high speed assumption. We restrict
  ourselves to dimension $1$ to have stable solitons (in
~\cite{IaLe14}, any dimension was allowed). The overall proof strategy is
  similar, but in our case we need to
  perform several technical refinements which include in particular
  working with localized
  momenta and modulated waves.
  
  In addition,  we are introducing
  the technical artefact consisting into introducing arbitrary
  constants in the definition~\eqref{eq:Stilde} of the global
  action. This is a new feature for this type of analysis, which is
  quite surprising as usually such a flexibility is not allowed by the
  algebra of the problem. 
\end{remark}

The scheme of the proof is inspired by the one developed
for the study of multi-solitons in
scalar nonlinear Schr\"odinger equations in
\cite{CoLe11,CoMaMe11,MaMe06,Me90} (see also~\cite{BeGhLe14} for a
similar approach applied to Klein-Gordon equations). It consists in solving
\eqref{eq:nls} backward in time, taking as final data a couple of
solitary waves 
$\begin{pmatrix}R_1(T^{n}),R_2(T^{n})\end{pmatrix}$, for an increasing
sequence of times $T^{n}\rightarrow +\infty$. Thus we get a sequence
$\begin{pmatrix}u^n_1,u^n_2\end{pmatrix}$ of solutions to
\eqref{eq:nls} on a time interval $(-\infty,T^n]$ such that
$\begin{pmatrix}u^n_1(T^n),u^n_2(T^n)\end{pmatrix}=\begin{pmatrix}R_1(T^n),R_2(T^n)\end{pmatrix}$. We
then have to prove the existence of a time $T_{0}$, independent of $n$
such that for $n$ large enough,
$\begin{pmatrix}u^n_1,u^n_2\end{pmatrix}$ is close to
$\begin{pmatrix}R_1,R_2\end{pmatrix}$ on $[T_{0},T^n]$. The key tools at hand to prove Theorem~\ref{theorem} are 
\begin{itemize}
\item uniform in $n$ estimates
  \[
  \forall t\in [T_{0},T^n],\
  \left\|\begin{pmatrix}u^n_1(t)\\u^n_2(t)\end{pmatrix}- \begin{pmatrix}R_{1}(t)\\R_{2}(t)\end{pmatrix}\right\|_{H^1\times
    H^1}\leq C  e^{-\sqrt{\omega_{*}}v_{*}t},
  \]
\item a compactness argument that gives the existence of $\begin{pmatrix}u_{1}^0,u_{2}^0\end{pmatrix} \in H^1(\R)\times H^1(\R)$ such that $\begin{pmatrix}u_{1}^n,u_{2}^n\end{pmatrix}$ converges strongly in $H^s(\R)$ ($s\in [0,1)$) towards $\begin{pmatrix}u_{1}^0,u_{2}^0\end{pmatrix}$.
\end{itemize}

\begin{remark}
  The method used to obtain Theorem~\ref{theorem} is a powerful tool
  to obtain sharp existence results for multi-solitons composed of
  ground states. Another approach relying on a fixed point argument has
  been developed for nonlinear Schr\"odinger equations in
~\cite{LeLiTs15,LeTs14}. This approach is very flexible and allows to
  prove existence of solutions more complicated than the ones in
  Theorem~\ref{theorem} like infinite trains of solitons or multi-kinks. The main
  drawback is that it always requires a large speed
  assumption. 
\end{remark}

We also have tested numerically  if the multi-speed solitary wave configuration was
stable provided the starting waves are well-ordered and
well-separated. In other words, we took the interaction to be small
at the origin and the composing waves going away from each other.
With such a well-prepared initial
configuration, we remain close to a similar configuration in large
time. This suggests that the multi-speed solitary waves are stable (no
matter the coupling parameter). Note that this is expected due to
the fact that each wave taken individually is stable. We have
however no theoretical mean to verify this conjecture. Similar
difficulties arise in the analysis of the stability for multi-solitons in nonlinear
scalar Schr\"odinger equations (see e.g.~\cite{MaMeTs06}). 

\subsection{The numerical experiments}

We solve the system~\eqref{eq:nls} in one dimension by adapting the
time-splitting spectral method described in~\cite{BJM03}. 
This method is unconditionally stable, time reversible, of spectral-order accuracy in space and second-order accuracy in time, and it conserves the discrete total mass
\cite{Bao04}. One can refer to~\cite{AnBaBe13} for other possible schemes and their properties. 

We will also compute the real valued ground state (minimizer of the energy on fixed
$L^2$ mass constraints) of the system~\eqref{eq:snls} using a normalized gradient flow approach. This will be used to make the comparison between the outcome of the interaction between two
solitary waves and a solitary wave with profile $(\phi_1,\phi_2)$.

As already mentioned, the experiment consists in taking as initial data the initial data of two solitary waves facing each other, each on one component.

We considered four cases, the first one being the integrable case, where
we expect the solitons after the interaction to move with the same
velocity and amplitude. Apart when $\mu_1=\mu_2=\beta$, the system is
not integrable, hence we do not expect pure elastic interaction
between solitons. However, there are still regimes where we expect the
outcome of interaction between solitons to be also a multi-speeds
solitary wave, different from the input at two level: first, there are
modifications in the speeds and amplitudes of the composing
solitons. Second, there is  a loss of a bit of energy, mass and
momentum into a small dispersive remainder. In certain cases, we have been able to
identify the profiles of the outcome of the interactions as ground
states of the stationary system~\eqref{eq:snls}.

The rest of this paper is organized as follows. In Section
\ref{sec:theorem}, we prove Theorem~\ref{theorem} assuming uniform
estimates. In Section~\ref{sec:uniform}, we prove the uniform estimates. The numerical methods are described in Section
\ref{sec:scheme}, and the numerical experiments are presented in
Section~\ref{sec:experiments}. Appendix~\ref{appendix} contains the proof of a modulation result.

\section{Existence of multi-speed solitary waves}
\label{sec:theorem}

This short section is devoted to the proof of Theorem~\ref{theorem},
assuming uniform estimates proved in the next section. 

In this section and in the next one, we assume that $\mu_1,\mu_2>0$ and
$\beta\in\setminus\{0\}$ are fixed constants and that we are given for
$j=1,2$ soliton parameters $\omega_j,v_j,x_j,\gamma_j\in\R$. Denote by
$Q_{\omega_j}$ and $R_j$ the corresponding profile and soliton.

We make the assumption that 
\begin{equation}
  \label{eq:assumption}
  0<v_1=-v_2.
\end{equation}
Since~\eqref{eq:nls} is Galilean invariant, this
assumption can be done without loss of generality. This will simplify
calculations later on.

Note that it follows from classical arguments (see~\cite{Ca03}) that
the Cauchy problem for~\eqref{eq:nls} is globally well-posed in the
energy space $H^1(\R)\times H^1(\R)$ and also in $L^2(\R)\times
L^2(\R)$. In particular, for any initial data $(u_1^0,u_2^0)\in
H^1(\R)\times H^1(\R)$ there exists a unique global solution
$(u_1,u_2)$ of~\eqref{eq:nls} in $\mathcal C(\R,H^1(\R)\times H^1(\R))\cap \mathcal C^1(\R,H^{-1}(\R)\times H^{-1}(\R))$.

Let $(T^n)$ be an increasing sequence of times such that
$T^n\to+\infty$ as $n\to+\infty$. Let
$(u^n_1,u^n_2)$
be the sequence of solutions to~\eqref{eq:nls} defined by solving
\eqref{eq:nls} backward on  $(-\infty,T^n]$
with final data $(u^n_1,u^n_2)(T^n)=\begin{pmatrix}R_1,R_2\end{pmatrix}(T^n).$
The proof of Theorem~\ref{theorem} then relies on the following two
ingredients.

First, we have uniform estimates on the distance between the sequence
$\begin{pmatrix}u^n_1,u^n_2\end{pmatrix}$ and the multi-speed solitary wave
profile $\begin{pmatrix}R_1,R_2\end{pmatrix}.$
\begin{proposition}
  [Uniform estimates]
  \label{prop:uniform}
  There exist $T_0>0$, $n_0\in\mathbb N$ such that, for all $n\geq
  n_0$ and for all $t\in[T_0,T^n]$ we have 
  \[
  \left\|\begin{pmatrix}u_{1}^n\\u_{2}^n\end{pmatrix}(t)
    -\begin{pmatrix}R_{1}\\R_{2}\end{pmatrix}(t) \right\|_{\hhu}
  \leq   e^{-\sqrt{\omega_{*}}v_{*}t},
  \]
  where, as in Theorem~\ref{theorem}, $\omega_{*}=\frac{1}{2304}\min\{\omega_{1},\omega_{2}\}$ and $v_{*}=\abs{v_{1}-v_{2}}$.

\end{proposition}
The proof of Proposition~\ref{prop:uniform} is rather involved and we
postpone it to Section~\ref{sec:uniform}.

The next ingredient is a compactness result on the initial data
$\begin{pmatrix}u_{1}^n,u_{2}^n\end{pmatrix}(T_0)$. This result was
already present in this form in~\cite{IaLe14} and we recall it without proof.
\begin{proposition}
  [Compactness]
  There exists $\begin{pmatrix}u_{1}^0,u_{2}^0\end{pmatrix}\in
  H^1(\R)\times H^1(\R)$  such that, up to a subsequence,
  $\begin{pmatrix}u_{1}^n,u_{2}^n\end{pmatrix}(T_0)$ converges strongly towards $\begin{pmatrix}u_{1}^0,u_{2}^0\end{pmatrix}$ in
  $H^s(\R)\times H^s(\R)$ for all $s\in [0,1)$.
\end{proposition}

With these two ingredients in hand, we can now conclude the proof of
Theorem~\ref{theorem}.

\begin{proof}[Proof of Theorem~\ref{theorem}]
  Let $(u_1,u_2)$ be the solution on $\R$ of the Cauchy problem
~\eqref{eq:nls} with initial data $(u_1^0,u_2^0)$ at $t=T_0$. By
  $H^1(\R)\times H^1(\R)$ boundedness and local
  well-posedness of the Cauchy problem in   $H^s(\R)\times H^s(\R)$ for
  all $s\in [0,1)$, we have weak convergence in  $H^1(\R)\times H^1(\R)$
  of $(u_1^n,u_2^n)(t)$ towards $(u_1,u_2)(t)$ for any $t\in\R$. Combined with the
  uniform estimates of Proposition~\ref{prop:uniform}, this implies for
  all $t\in[T_0,+\infty)$ that 
  \[
  \norm*{\begin{pmatrix}u_1\\u_2\end{pmatrix}(t)-\begin{pmatrix}R_1\\R_2\end{pmatrix}(t)}_{H^1\times
    H^1}
  \leq 
  \liminf_{n\to+\infty}
  \norm*{\begin{pmatrix}u_1^n\\u_2^n\end{pmatrix}(t)-\begin{pmatrix}R_1\\R_2\end{pmatrix}(t)}_{H^1\times
    H^1}
  \leq 
  C e^{-\sqrt{\omega_{*}}v_{*}t}.
  \]
  This concludes the proof of Theorem~\ref{theorem}.
\end{proof}

\section{Uniform estimates}
\label{sec:uniform}
This section is devoted to the proof of Proposition
\ref{prop:uniform}. In all this section, $T^n$ and
$\begin{pmatrix}u_{1}^n,u_{2}^n\end{pmatrix}$ are given as in the
beginning of Section~\ref{sec:theorem}.

\subsection{The bootstrap argument} 

We first reduce the proof of Proposition~\ref{prop:uniform} to the
proof of the following bootstrap result. 
\begin{proposition}[Bootstrap argument]
  \label{bootstrap}
  There exist $T_{0}>0$ and $n_{0}\in \N$ such that for all $n\geq n_{0}$
  and  for any $t_{0}\in [T_{0},T^n]$ the following property is
  satisfied. 
  If for all $t\in [t_{0},T^n]$ we have 
  \begin{equation}
    \label{BH1} 
    \left\|\begin{pmatrix}u_{1}^n\\u_{2}^n\end{pmatrix}(t) -\begin{pmatrix}R_{1}\\R_{2}\end{pmatrix}(t) \right\|_{\hhu}\leq   e^{-\sqrt{\omega_{*}}v_{*}t},
  \end{equation}
  then for all $t\in [t_{0},T^n]$ we have 
  \begin{equation}
    \label{BH1/2}
    \left\|\begin{pmatrix}u_{1}^n\\u_{2}^n\end{pmatrix}(t) -\begin{pmatrix}R_{1}\\R_{2}\end{pmatrix}(t) \right\|_{\hhu}\leq  \frac{1}{2} e^{-\sqrt{\omega_{*}}v_{*}t}.
  \end{equation}
\end{proposition}

The proof of Proposition~\ref{bootstrap} will occupy us for most of the
rest of this section. We divided it into several steps. We first
perform a geometrical decomposition of the sequence $(u_1^n,u_2^n)$ onto
the manifold of multi-speed solitary waves in
order to obtain orthogonality conditions. We then introduce an
action-like functional, which turns out to be coercive due to our
orthogonality conditions. This functional is not a conserved quantity,
but since it is made with localized conservations laws it is almost
conserved. Using that property and a control on the geometrical
modulation parameters, we are able to conclude the proof of Proposition~\ref{bootstrap}.

Before going on with the details of the proof of Proposition
\ref{bootstrap}, let us show how it implies Proposition
\ref{prop:uniform}.

\begin{proof}[Proof of Proposition~\ref{prop:uniform}]
  Since we have
  $(u_1^n,u_2^n)(T^n)=(R_1^n,R_2^n)(T^n)$ at the final time $T^n$,  by continuity there exists a
  minimal 
  time $t_0$ such that for all $t\in[t_0,T^n]$ we have
  \begin{equation}
    \label{eq:38}
    \left\|\begin{pmatrix}u_{1}^n\\u_{2}^n\end{pmatrix}(t) -\begin{pmatrix}R_{1}\\R_{2}\end{pmatrix}(t) \right\|_{\hhu}\leq   e^{-\sqrt{\omega_{*}}v_{*}t}.
  \end{equation}
  We prove that $t_0=T_0$ by contradiction. 
  Assume that $t_0>T_0$. By Proposition~\ref{bootstrap}, for all
  $t\in[t_0,T^n]$ we have 
  \begin{equation*}
    \left\|\begin{pmatrix}u_{1}^n\\u_{2}^n\end{pmatrix}(t) -\begin{pmatrix}R_{1}\\R_{2}\end{pmatrix}(t) \right\|_{\hhu}\leq \frac12  e^{-\sqrt{\omega_{*}}v_{*}t}.
  \end{equation*}
  Therefore by continuity there exists $t_{00}<t_0$ such that on
  $[t_{00},T^n]$ estimate~\eqref{eq:38} is satisfied. This however
  contradicts the minimality of $t_0$. Hence $t_0=T_0$ and this
  concludes the proof. 
\end{proof}

For the rest of Section~\ref{sec:uniform}, $T_0>0$ and $n_0\in\mathbb
N$ will be large enough fixed numbers, 
and  we assume the existence of
$t_0\geq T_0$ such that for all $t\in [t_{0},T^n]$ the bootstrap
assumption~\eqref{BH1} is verified, i.e.  we have 
\begin{equation}
  \label{eq:BH1-bis}
  \left\|\begin{pmatrix}u_{1}^n\\u_{2}^n\end{pmatrix}(t) -\begin{pmatrix}R_{1}\\R_{2}\end{pmatrix} (t)\right\|_{\hhu}\leq   e^{-\sqrt{\omega_{*}}v_{*}t}.
\end{equation}
Our final goal is now to prove that in fact~\eqref{BH1/2} holds for  all $t\in [t_{0},T^n]$.

\subsection{Modulation}

Let us start with a decomposition lemma for our sequence of
approximated multi-speed solitary waves.  

\begin{lemma}[Modulation]
   \label{lem:modulation}
  There exist $C>0$ and $\calC^1$ functions 
  \begin{equation*}
    \tilde{\omega}_{j}:[t_{0},T^n]\rightarrow (0,+\infty),\quad
    \tilde{x}_{j}:[t_{0},T^n]\rightarrow \R,\quad
    \tilde{\gamma}_{j}:[t_{0},T^n]\rightarrow \R,\quad j=1,2,
  \end{equation*}
  such that if for $j=1,2$ we denote by $\tilde{R}_{j}$ the modulated wave
  \begin{equation}
    \label{eq=Rtilde}
    \tilde{R}_{j}(t,x)=e^{i\left(\frac{1}{2}v_{j} \cdot x+\tilde{\gamma}_{j}(t)\right)}\frac{1}{\sqrt{\mu_{j}}}Q_{\tilde{\omega}_{j}(t)}(x-\tilde{x}_{j}(t)),
  \end{equation}
  then for all $t\in [t_{0},T^n]$ the functions defined by
  \begin{equation*}
    \begin{pmatrix} \eps_{1}\\\eps_{2}\end{pmatrix}(t)=\begin{pmatrix} u_{1}^n\\u_{2}^n\end{pmatrix}(t)-\begin{pmatrix} \tilde{R}_{1}\\\tilde{R}_{2}\end{pmatrix}(t)
  \end{equation*}
  satisfy for $j=1,2$ and for all $t\in[t_0,T^n]$ the orthogonality conditions
  \begin{equation}
    \label{eq:orth}
    \psld{\eps_{j}(t)}{\tilde{R}_{j}(t)}=\psld{\eps_{j}(t)}{i \tilde{R}_{j}(t)}=\psld{\eps_{j}(t)}{\partial_{x}\tilde{R}_{j}(t)}=0.
  \end{equation}
  Moreover, for all $t\in [t_{0},T^n]$, we have
  \begin{multline}
    \label{eq:estmod}
    \sum_{j=1}^2\left(\abs{\partial_{t}\tildo_{j}(t)}^2+\abs{\partial_{t}\tilde{x}_{j}(t)-v_{j}}^2+\abs*{\partial_{t}\tilde{\gamma}_{j}(t)+\frac{v_{j}^2}{4}-\tildo_{j}(t)}^2\right)
    \\
    \leq C\left\|{\begin{pmatrix} \eps_{1}\\\eps_{2}\end{pmatrix}(t)}\right\|^2_{\hhu}+Ce^{-3\sqrt{\om}_{*}v_{*}t}.
  \end{multline}
\end{lemma}

\begin{remark}
  \label{remarque-modulation}
  It is to be noticed that estimate~\eqref{eq:estmod} clearly implies
  that, for $T_{0}$ large enough, for $j=1,2$, and for all $ t\in
  [t_{0},T^n]$, we have
  \[
  \tildx_{j}(t)\geq
  \frac{v_{*}}{2\sqrt{2}}t >2L\textrm{ and }\ \tildo_{j}(t)\geq 1152
  \om_{*}.
  \]
  Moreover, a better estimate can be obtained for $\tildo_{j}$ and will be stated later on in Lemma~\ref{variation-omega1}.
\end{remark}

This type of modulation result is classical in the literature dealing
with solitary waves of nonlinear dispersive equations (see e.g. the
fundamental paper of Weinstein~\cite{We85} for an early version or
\cite{MaMeTs06} for a recent approach). Its proof consists essentially in
the application of the implicit function theorem combined with the use
of the evolution equation to find equation~\eqref{eq:estmod} for the
evolution of the modulation
parameters. We refer to the appendix for the details of the proof.

\subsection{Energy estimates and coercivity}

In this subsection, we analyze the different quantities that are conserved or almost-conserved in our coupled-vectorial problem. Remember that in the case of the scalar equation~\eqref{eq:nls-j}  the energy, mass and momentum, defined as follows, are conserved along the flow of~\eqref{eq:nls-j}:
\[
E(u,\mu_{j}):=\frac{1}{2}\norm{\partial_{x}u}_{L^2}^2-\frac{\mu_{j}}{4}\norm{u}_{L^4}^4,\quad
M(u):=\frac{1}{2}\norm{u}_{L^2}^2,\quad
P(u):=\frac{1}{2}\Im{\int_{\R}u\overline{\partial_{x}u}dx}.
\]
The solution $Q_{\om}$ we chose of equation~\eqref{eq:Q-omega} is
known to be the unique positive radial ground state of the action $S:=E(\cdot,1)+\om
M$. 
Consequently, each soliton $R_{j}$ defined by 
\eqref{eq:soliton-j} is a critical point of the scalar functional $S_{j}$ defined by 
\begin{equation}
  \label{eq:Si1}
  S_{j}:=E(\cdot,\mu_{j})+\left(\om_{j}+\frac{v_{j}^2}{4}\right)M+v_{j} P.
\end{equation}
Coercivity properties of linearizations of $S_{j}$-like functionals are the key tool of the analysis of multi-solitons interaction (see for example~\cite{IaLe14,MaMe06}). 

In the vectorial case we are interested in here, the coupled system
\eqref{eq:nls} admits its own conservation laws. In particular, the
mass of each component is preserved, as in the scalar case. However,
the coupling does not preserve conservation of the scalar energy and
momentum for each component and we only have conservation of the total
energy (made of individual energies plus a coupling term) and total
momentum (sum of the scalar momenta). More precisely, total energy,
total momentum, and scalar masses of the whole system (defined as follows) are conserved quantities for the flow of system~\eqref{eq:nls}:

\begin{gather*}
  \calE \begin{pmatrix}u_1\\u_2\end{pmatrix}:=E(u_{1},\mu_{1})+E(u_{2},\mu_{2})-\frac{\beta}{2}\int_{\R}\abs{u_{1}}^2\abs{u_{2}}^2dx,
  \\
  \calP\begin{pmatrix}u_1\\u_2\end{pmatrix}:=P(u_{1})+P(u_{2}),
  \quad
  M_{j}\begin{pmatrix}u_1\\u_2\end{pmatrix}:=M(u_{j}),\quad j=1,2.
\end{gather*}

In order to use the conservation of the momentum as in the scalar
case, we here need to localize the momentum of each soliton, as was
done in~\cite{CoLe11,CoMaMe11} for the scalar mass and momentum. Note
that this was not needed for the analysis in~\cite{IaLe14}. Let us
define the cut-off functions
\[
\chi^{1}_{L}(x)=\chi\left(\frac{x}{L}\right),\quad
\chi_{L}^2=1-\chi_{L}^1,
\] 
where $L>0$ is arbitrary but fixed and $\chi$ is a $\calC^3$ function such that 
\begin{equation*}
  0\leq \chi\leq 1 \text{ on } \R,\quad \chi(x)=0 \text{ for } x\leq -1,\quad \chi(x)=1 \text{ for } x> 1,\quad \chi'\geq 0 \text{ on } \R, 
\end{equation*}
and satisfies, for some positive constant $C$ and for all $x \in \R$
the estimates
\[
(\chi'(x))^2\leq C \chi(x),\quad
(\chi''(x))^2\leq C\chi'(x).
\]
Localized momenta $\calPlocj$ are defined by:
\begin{equation*}
  \calPlocj\begin{pmatrix}u_1\\u_2\end{pmatrix} =\frac{1}{2}\
  \Im \int_{\R}
  \left(u_{1}\overline{\partial_{x}u_{1}}+u_{2}\overline{\partial_{x}
      u_{2}}\right)\chi_{L}^{j} dx,\quad j=1,2.
\end{equation*}
Remark that $\mathcal P=\calPlocu+\calPlocd$. Note that since we are assuming~\eqref{eq:assumption} the
momenta above defined are localized around each composing solitary
wave of the profile. The advantage of having made assumption
\eqref{eq:assumption} is that the cut-off does not depend on
time. This will simplify our next calculations. 

In the sequel, we are interested in the following global action:
\begin{multline}
  \label{eq:Stilde}
  \mathcal S \begin{pmatrix}u_1\\u_2\end{pmatrix} =\calE \begin{pmatrix}u_1\\u_2\end{pmatrix}+\sum_{j=1,2}\left(\tildo_{j}(t)+\frac{v_{j}^2}{4}\right)M_{j} \begin{pmatrix}u_1\\u_2\end{pmatrix} +
  \sum_{j=1,2}v_{j} \calPlocj \begin{pmatrix}u_1\\u_2\end{pmatrix} \\+C_{1}(\beta,v_{1}) M_{2} \begin{pmatrix}u_1\\u_2\end{pmatrix} +C_{2}(\beta,v_{2})M_{1} \begin{pmatrix}u_1\\u_2\end{pmatrix},
\end{multline}
where $C_{j}(\beta,v_{j})$ are positive constants depending only on
$\beta$ and $v_{j}$ and whose exact values will be decided later
on. Note that the action implicitly depends on $t$ via $\tilde\omega_j$.
It is to be noted that, in this work, we have  the freedom to add
these two coupled-mass terms that do not appear in the usual
definition of $\mathcal S$-like functionals. This is a key point in our
analysis. 

Let us now state in the following lemma several estimates related to the
localization of $\tildR_{1}$ and $\tildR_{2}$ and which will be of great
use in the sequel.

\begin{lemma}
  \label{localisation-estimates}
  For $j=1,2$, if $T_{0}$ is large enough, then for all $t\in
  [t_{0},T^n]$ and for all $x\in \R$ we have
  \begin{align}
    \left(\abs{\tildR_{j}(t,x)}+\abs{\partial_{x}\tildR_{j}(t,x)}\right)\chi_L^{3-j}(x)
    &\leq 
    C(1+\abs{v_{j}})
    e^{-3\sqrt{\omega_*}v_*t}
    e^{-\sqrt{\omega_*}|x|},
    \label{eq:2}
    \\
    \prod_{k=1,2}\left(\abs{\tildR_{k}(t,x)}+\abs{\partial_{x}\tildR_{k}(t,x)}\right)
    &\leq
    C(1+\abs{v_{1}}+\abs{v_{2}})^2
    e^{-3 \sqrt{\omega_{*}} v_*t}
    e^{-\sqrt{\omega_{*}}\abs{x}}.
    \label{eq:3}
  \end{align}
\end{lemma}

Lemma~\ref{localisation-estimates} follows
from the support properties of the
cut-off function and
the exponential
localization of the solitons profiles. Indeed, recall
that in fact, the profile $Q=Q_1$ is explicitly known
\[
Q(x)=2\operatorname{sech}(x),
\]
and it follows  that $Q$ and its derivatives are exponentially
decaying, i.e. for any $\eta<1$ we have 
\[
|Q_{xx}|+|Q_x|+|Q|\leq C_\eta e^{-\eta|x|}.
\]

\begin{proof}[Proof of Lemma~\ref{localisation-estimates}]
  We prove only~\eqref{eq:2}, the proof of~\eqref{eq:3} following from
  similar (simpler) arguments.  
  For simplicity in notation, assume $j=1$, the case $j=2$ being
  perfectly symmetric. 
  Due to the exponential decay of the soliton profiles, we have 
  \[
  \abs{\tildR_{1}(t,x)}\leq C
  e^{-\frac34\sqrt{\tilde\omega_1}|x-\tilde x_1|}
  \]
  The cut-off function $\chi_L^2$ is supported on $(-\infty,L]$, and
  since for $T_0$ large enough $\tilde x_1>2L$, for
  $x\in (-\infty,L]$ we have (we recall here Remark~\ref{remarque-modulation})
  \[
  |x-\tilde x_1|\geq |x|, \quad\text{and} \quad   |x-\tilde x_1|\geq\frac12|\tilde x_1|.
  \]
  As a consequence, we have
  \[
  \abs{\tildR_{1}(t,x)}\chi_L^{2}(x)\leq C
  e^{-\frac14\sqrt{\tilde\omega_1}|\tilde x_1|}
  e^{-\frac14\sqrt{\tilde\omega_1}|x|}.
  \]
  In addition, as noticed in Remark~\ref{remarque-modulation}, we have 
  \[
  \tilde x_1\geq \frac {v_*}{2\sqrt{2}}t,\quad\text{and}\quad
  \tilde\omega_1\geq 1152\omega_*,
  \]
  which implies
  \[
  \abs{\tildR_{1}(t,x)}\chi_L^{2}(x)\leq C
  e^{-3\sqrt{\omega_*}v_*t}
  e^{-\sqrt{\omega_*}|x|}.
  \]
  The derivative $\partial_x \tildR_{1}$ is treated in the same way,
  with the only difference that $|v_1|$ now appears in the estimate,
  due to the term $e^{i\frac12v_1\cdot x}$ in the definition
~\eqref{eq=Rtilde} of $\tilde R_1$. This finishes the proof.
\end{proof}

\begin{lemma}
  [Expansion of the global action $\mathcal S$]
  \label{expansion}
  For all $t\in [t_{0},T^n]$ we have
  \begin{equation}
    \label{tildS:expansion}
    \mathcal S \begin{pmatrix}u_1\\u_2\end{pmatrix} =
    \mathcal S \begin{pmatrix}\tildR_1\\\tildR_2\end{pmatrix} +
    \mathcal H \begin{pmatrix}\eps_1\\\eps_2\end{pmatrix} 
   +\calO(e^{-3\sqrt{\omega_*}v_*t}),
  \end{equation}  
  \begin{equation*}
    \mathcal H\begin{pmatrix}\eps_1\\\eps_2\end{pmatrix} 
    =\mathcal H_{\mathrm{free}}\begin{pmatrix}\eps_1\\\eps_2\end{pmatrix} 
    +\mathcal H_{\mathrm{coupled}}\begin{pmatrix}\eps_1\\\eps_2\end{pmatrix} ,
  \end{equation*}
  with 
  \begin{multline*}
    \mathcal H_{\mathrm{free}}\begin{pmatrix}\eps_1\\\eps_2\end{pmatrix} 
    =\sum_{j=1,2}\bigg(\frac{1}{2}\norm{\partial_{x}\eps_{j}}_{L^2}^2+\frac{1}{2}\left(\tildo_{j}+\frac{v_{j}^2}{4}\right)\norm{\eps_{j}}_{L^2}^2-\mu_{j}\int_{\R}\abs{\eps_{j}}^2\abs{\tildR_{j}}^2dx\\
    -\frac{\mu_{j}}{2}\Re
    \int_{\R}\eps_{j}^2\overline{\tildR_{j}^2}dx+\frac{1}{2} v_{j}\
    \Im\int_{\R}\eps_{j}\overline{\partial_{x}\eps_{j}}
    \chi_{L}^{j}
    dx\bigg)
  \end{multline*}
  and
  \begin{multline*}
    \mathcal
    H_{\mathrm{coupled}}\begin{pmatrix}\eps_1\\\eps_2\end{pmatrix} 
    =C_{1}(\beta,v_{1})\norm{\eps_{2}}_{L^2}^2+C_{2}(\beta,v_{2})\norm{\eps_{1}}_{L^2}^2
    \\
    -\frac{\beta}{2}\int_{\R}\left(\abs{\eps_{1}}^2\abs{\tildR_{2}}^2+\abs{\eps_{2}}^2\abs{\tildR_{1}}^2
   \right) dx
    \\
    +\frac{1}{2}v_{1}\Im \int_{\R} \eps_{2}\overline{\partial_{x}\eps_{2}}\chi_{L}^1dx+\frac{1}{2}v_{2}\Im \int_{\R} \eps_{1}\overline{\partial_{x}\eps_{1}}\chi_{L}^2dx.
  \end{multline*}
\end{lemma}

\begin{proof}[Proof of Lemma~\ref{expansion}]
  First note that, according to the definition~\eqref{eq:Stilde} of $\mathcal S$, 
  \begin{multline}
    \label{eq:Stilde2}
    \mathcal S \begin{pmatrix}u_1\\u_2\end{pmatrix} = \sum_{j=1,2}
    \tilde S_{j}(u_{j})  + v_{j}\cdot \left( \calPlocj \begin{pmatrix}u_1\\u_2\end{pmatrix}- P(u_i)\right) \\
    + C_{1}(\beta,v_{1}) M(u_{1})+ C_{2}(\beta,v_{2}) M(u_{2})
    -\frac{\beta}{2}\int_{\R} |u_{1}|^2|u_{2}|^2 dx,
  \end{multline}
  where $\tilde S_j$ denote the same functional as $S_j$ (see the
  definition~\eqref{eq:Si1}) with $\tilde
  \omega_j$ instead of $\omega_j$. 
  For $j=1,2$, let us expand $u_{j}(t)=\tildR_{j}(t)+\eps_{j}(t)$ in the
  expression of $\tilde  S_{j}$. 
  A simple computation leads to:
  \[
  \tilde S_{j}(\tildR_{j}+\eps_{j})=\tilde S_{j}(\tildR_{j})+\tilde
  S'_{j}(\tildR_{j})\eps_{j}+\dual{\tilde
    S''_{j}(\tildR_{j})\eps_{j}}{\eps_{j}}+\mathcal O(\norm{\eps_j}_{H^1}^3)
  \]
  Now, as $\tildR_{j}$ is a critical point of the scalar functional
  $\tilde S_j$, we have
  \[
  \tilde S'_{j}(\tilde R_{j})=0,
  \]
  and thus 
  \begin{equation}
    \label{eq:Si}
    \tilde S_{j}(\tildR_{j}+\eps_{j})=\tilde S_{j}(\tildR_{j})+\dual{\tilde S''_{j}(\tildR_{j})\eps_{j}}{\eps_{j}}+\mathcal O(\norm{\eps_j}_{H^1}^3),
  \end{equation}
  where 
  \begin{multline}
    \label{crochet}
    \dual{\tilde S''_{j}(\tildR_{j})\eps_{j}}{\eps_{j}}=\frac{1}{2}\norm{\partial_{x}\eps_{j}}_{L^2}^2+\frac{1}{2}\left(\tilde\om_{j}+\frac{v_{j}^2}{4}\right)\norm{\eps_{j}}_{L^2}^2+\frac{1}{2}v_{j}\Im \int_{\R} \eps_{j}\overline{\partial_{x}\eps_{j}}dx \\-\frac{\mu_{j}}{2}\Re \int_{\R} \eps_{j}^2\overline{\tildR_{j}^2}dx-\mu_{j}\int_{\R}|\eps_{j}|^2|\tildR_{j}|^2 dx.
  \end{multline}
  Let us now develop the remaining terms in~\eqref{eq:Stilde2}. As far as the momentum part is concerned, we write the expansion for $j=1$ for simplicity:
  \begin{multline}
    \label{momentum-expansion}
    \calPlocu \begin{pmatrix}\tildR_{1}+\eps_{1}\\\tildR_{2}+\eps_{2}\end{pmatrix}
    - P(\tildR_{1}+\eps_{1})
    \\
    \shoveleft=-\Im \int_{\R} \tildR_{1}\overline{\partial_{x}\tildR_{1}}\chi_{L}^2dx
    +\Im \int_{\R} \tildR_{2}\overline{\partial_{x}\tildR_{2}}\chi_{L}^1dx\\
    -\Im \int_{\R} \left(\tildR_{1}\overline{\partial_{x}\eps_{1}}+\eps_{1}\overline{\partial_{x}\tildR_{1}}\right)\chi_{L}^2 dx
    +\Im \int_{\R} \left(\tildR_{2}\overline{\partial_{x}\eps_{2}}+\eps_{2}\overline{\partial_{x}\tildR_{2}}\right)\chi_{L}^1dx\\
    -\Im \int_{\R} \eps_{1}\overline{\partial_{x}\eps_{1}}\chi_{L}^2 dx +\Im \int_{\R} \eps_{2}\overline{\partial_{x}\eps_{2}}\chi_{L}^1dx.
  \end{multline}
  Concerning the $\beta$ coupling part in~\eqref{eq:Stilde2}, we get:
  \begin{multline}
    \label{betacoupling-expansion}
    \int_{\R} |\tildR_{1}+\eps_{1}|^2|\tildR_{2}+\eps_{2}|^2 dx 
    \\\shoveleft= \int_{\R} |\tildR_{1}|^2|\tildR_{2}|^2dx 
    +2\Re \int_{\R} \left(|\tildR_{1}|^2\overline{\tildR_{2}}\eps_{2}+|\tildR_{2}|^2\overline{\tildR_{1}}\eps_{1}\right)dx\\
    +\int_{\R} \left(\abs{\eps_{1}}^2\abs{\tildR_{2}}^2+\abs{\eps_{2}}^2\abs{\tildR_{1}}^2+ 4 \Re (\eps_{1}\tildR_{1})\Re (\eps_{2}\tildR_{2})\right)dx\\
    +2 \int_{\R} \left(|\eps_{2}|^2\Re(\eps_{1}\overline{\tildR_{1}})+|\eps_{1}|^2\Re(\eps_{2}\overline{\tildR_{2}})\right)dx
    +\int_{\R} |\eps_{1}|^2|\eps_{2}|^2dx.
  \end{multline}
  Finally, the extra-masses terms in~\eqref{eq:Stilde2} expand into 
  \begin{equation}
    \label{mass-expansion}
    M(\tildR_{j}+\eps_{j})=\frac{1}{2}\norm{\tildR_{j}+\eps_{j}}_{L^2}^2=M(\tildR_{j})+ \psld{\tildR_{j}}{\eps_{j}} +M(\eps_{j})=M(\tildR_{j})+ M(\eps_{j}),
  \end{equation}
  where we have used the orthogonality conditions~\eqref{eq:orth} to
  obtain the last equality.
  In~\eqref{momentum-expansion} and~\eqref{betacoupling-expansion},
  all terms containing a product of solitons or cut-off functions with
  different indices/exponents are of order $\mathcal
  O(e^{-3\sqrt{\omega_*}v_*t})$ by Lemma
~\ref{localisation-estimates}. The terms containing a degree $3$ or
  higher term in $(\eps_1,\eps_2)$ are also  of order $\mathcal
  O(e^{-3\sqrt{\omega_*}v_*t})$ by the bootstrap assumption~\eqref{eq:BH1-bis}.
  Therefore, 
  gathering~\eqref{eq:Stilde2}-\eqref{eq:Si}-\eqref{crochet}-\eqref{momentum-expansion}-\eqref{betacoupling-expansion}-\eqref{mass-expansion} together gives: 
  \[
  \mathcal
  S \begin{pmatrix}\tildR_1+\eps_{1}\\\tildR_{2}+\eps_{2}\end{pmatrix}
  = \mathcal S \begin{pmatrix}\tildR_1\\\tildR_{2}\end{pmatrix}
  +\mathcal H_{\mathrm{free}}\begin{pmatrix}\eps_{1}\\\eps_{2}\end{pmatrix} +\mathcal
  H_{\mathrm{coupled}}\begin{pmatrix}\eps_{1}\\\eps_{2}\end{pmatrix}
  +\mathcal
  O(e^{-3\sqrt{\omega_*}v_*t})
  .
  \]
  This
  concludes the proof.
\end{proof}

\begin{lemma}[Coercivity of $\mathcal H$]
  There exists $\lambda>0$   such that, for any $t_{0}\in [T_{0},T^n]$ and for all $t\in [t_{0},T^n]$:
  \begin{gather}
    \label{Hfree-coercivity}
    \mathcal H_{\mathrm{free}}\begin{pmatrix}\eps_{1}\\\eps_{2}\end{pmatrix} \geq 2\lambda \left(\norm{\eps_{1}}^2_{H^1}+\norm{\eps_2}_{H^1}^2\right),\\
    \label{Hcoupled-coercivity}
    \mathcal H_{\mathrm{coupled}}\begin{pmatrix}\eps_{1}\\\eps_{2}\end{pmatrix} \geq \lambda \left(\norm{\eps_{1}}_{L^2}^2+\norm{\eps_{2}}_{L^2}^2\right)-\lambda\left(\norm{\partial_{x}\eps_{1}}_{L^2}^2+\norm{\partial_{x}\eps_{2}}_{L^2}^2\right),
  \end{gather}
  and thus:  
  \begin{equation}
    \label{coercivity}
    \mathcal H\begin{pmatrix}\eps_{1}\\\eps_{2}\end{pmatrix} \geq \lambda
    \left\|\begin{pmatrix}\eps_{1}\\\eps_{2}\end{pmatrix} \right\|_{H^1\times
      H^1}^2.
  \end{equation}
\end{lemma}

\begin{proof}
  The proof of~\eqref{Hfree-coercivity} is 
  classical (see for example~\cite{IaLe14, MaMe06}) and we omit it.
  It remains to prove~\eqref{Hcoupled-coercivity}. It is readily seen
  that for $C_j(\beta,v_{j})$ large enough, we have
  \begin{multline*}
    C_{1}(\beta,v_{1})\norm{\eps_{2}}_{L^2}^2+C_{2}(\beta,v_{2})\norm{\eps_{1}}_{L^2}^2-\frac{\beta}{2}\int_{\R}\left(\abs{\eps_{1}}^2\abs{\tildR_{2}}^2+\abs{\eps_{2}}^2\abs{\tildR_{1}}^2\right)
    dx\\\geq \frac12\left(
      C_{1}(\beta,v_{1})\norm{\eps_{2}}_{L^2}^2+C_{2}(\beta,v_{2})\norm{\eps_{1}}_{L^2}^2\right).
  \end{multline*}
  As far as the momentum parts of $\mathcal H_{\mathrm{coupled}}$ are concerned,
  for example for $j=1$:
  \begin{align*}
    \frac{v_{1}}{2}\Im \int_{R}
    \eps_{1}\overline{\partial_{x}\eps_{1}}\chi_{L}^2 dx
    &\geq -\frac{v_{1}}{2}\int_{\R} \abs{\eps_{1}\partial_{x}\eps_{1}\chi_{L}^2} dx\\
    & \geq -\frac{v_{1}}{2} \norm{\eps_{1}}_{2}\norm{\partial_{x}\eps_{1}}_{2}\\
    &\geq
    -\frac{v_1^2}{4\lambda}\norm{\eps_{1}}_{L^2}^2-\lambda\norm{\partial_{x}\eps_{1}}_{L^2}^2.
  \end{align*}
  For $C(\beta,v_{j})$, $j=1,2$ large enough, i.e such that
  \[
  \sum_{j=1,2}
  \left(\frac12C(\beta,v_{j})-\frac{v_j^2}{4\lambda}\right)>\lambda,
  \]
  we thus have:
  \[
  \mathcal H_{\mathrm{coupled}}\begin{pmatrix}\eps_{1}\\\eps_{2}\end{pmatrix} \geq \lambda
  \left(\norm{\eps_{1}}_{L^2}^2+\norm{\eps_{2}}_{L^2}^2\right)-\lambda\left(\norm{\partial_{x}\eps_{1}}_{L^2}^2+\norm{\partial_{x}\eps_{2}}_{L^2}^2\right).
  \]
  Combining estimates~\eqref{Hfree-coercivity} and
~\eqref{Hcoupled-coercivity} on $\mathcal H_{\mathrm{free}}$ and
  $\mathcal H_{\mathrm{coupled}}$ finally gives~\eqref{coercivity}.
\end{proof}

\subsection{Almost-conservation of the localized momentum}

In this section, we investigate the conservation of the localized momentum and, inspired by~\cite{MaMe06}, we state the following lemma:
\begin{lemma}
  \label{conservation-momentum}
  There exists $C>0$ (independent of $L$)
  such that if $L$ and $T_{0}$ are large enough then for all $t\in
  [t_{0},T^{n}]$ and for $j=1,2$ we have
  \begin{equation}
    \label{cons-momentum}
    \left|\calPlocj \begin{pmatrix}u_1\\u_2\end{pmatrix}(t) - \calPlocj \begin{pmatrix}u_1\\u_2\end{pmatrix}(T^{n}) \right|\leq \frac{C}{L} e^{-2\sqrt{\omega_*}v_* t}.
  \end{equation}
\end{lemma}

\begin{proof}[Proof of Lemma~\ref{conservation-momentum}]
  Let us prove the lemma in the case $j=1$, the case $j=2$ being
  perfectly similar. We first compute the time derivative of $\calPlocu$:
  \begin{equation}
    \label{dtPloc}
    \frac{d}{dt} \calPlocu (t)= \frac{1}{2}\Im \int_{\R} \partial_{t}u_{1}\partial_{x}\overline{u_{1}}\chi\left(\frac{x}{L}\right)dx + \frac{1}{2}\Im \int_{\R} u_{1} \partial_{x,t}\overline{u_{1}}\chi\left(\frac{x}{L}\right) dx.
  \end{equation}
  Let us call term $A$ and $B$ respectively the first and the second term of the right hand side in equality~\eqref{dtPloc}:
  \[
  A=\frac{1}{2}\Im
  \int_{\R} \partial_{t}u_{1}\partial_{x}\overline{u_{1}}\chi\left(\frac{x}{L}\right)dx,\qquad
  B=\frac{1}{2}\Im \int_{\R}
  u_{1} \partial_{x}\partial_{t}\overline{u_{1}}\chi\left(\frac{x}{L}\right)
  dx.
  \]
  Using the fact that $u(t)=\begin{pmatrix}u_1(t)\\u_2(t)\end{pmatrix}$
  is a solution to~\eqref{eq:nls}, we readily get (formally, but this
  will be justified when integrating):
  \begin{align}
    \label{dtudxu}
    &\partial_{t} u_{1} \partial_{x} \overline{u_{1}} = i\partial_{xx} u_{1} \partial_{x}\overline{u_{1}}+i\mu_{1} \abs{u_{1}}^2u_{1}\partial_{x}\overline{u_{1}}+i\beta \abs{u_{2}}^2 u_{1}\partial_{x}\overline{u_{1}}\\
    \label{dxdtu}
    &\partial_{x}\partial_{t} \overline{u_{1}} = -i\partial_{xxx} \overline{u_{1}} - i \mu_{1}\partial_{x}\left(\abs{u_{1}}^2 \overline{u_{1}}\right)-i\beta \partial_{x}\left(\abs{u_{2}}^2\overline{u_{1}}\right).
  \end{align} 
  \textit{About term A:} Equation~\eqref{dtudxu} provides us with the following decomposition of term $A$:
  \begin{multline*}
    A=\frac{1}{2}\Re \int_{\R} \partial_{xx}u_{1}\partial_{x}\overline{u_{1}}\chi\left(\frac{x}{L}\right)dx + \frac{1}{2}\mu_{1}\Re \int_{\R} \abs{u_{1}}^2u_{1}\partial_{x}\overline{u_{1}}\chi\left(\frac{x}{L}\right)dx \\+ \frac{1}{2}\beta\Re \int_{\R} \abs{u_{2}}^2u_{1}\partial_{x}\overline{u_{1}}\chi\left(\frac{x}{L}\right)dx.
  \end{multline*}
  Integrating by part in each term of $A$ finally gives:
  \begin{multline}
    \label{finalA}
    A=-\frac{1}{4L}\int_{\R} \abs{\partial_{x}u_{1}}^2\chi'\left(\frac{x}{L}\right)dx-\frac{1}{8L}\mu_{1}\int_{\R} \abs{u_{1}}^4 \chi'\left(\frac{x}{L}\right)dx \\+  \frac{1}{2}\beta\Re \int_{\R} \abs{u_{2}}^2u_{1}\partial_{x}\overline{u_{1}}\chi\left(\frac{x}{L}\right)dx.
  \end{multline} 
  \textit{About term B:} Equation~\eqref{dxdtu} provides us with the following decomposition of term $B$: 
  \begin{multline*}
   B=-\frac{1}{2}\Re \int_{\R} u_{1}\partial_{xxx}\overline{u_{1}}\chi\left(\frac{x}{L}\right)dx -\frac{1}{2}\mu_{1}\Re\int_{\R}u_{1}\partial_{x}\left(\abs{u_{1}}^2\overline{u_{1}}\right)\chi\left(\frac{x}{L}\right)dx\\- \frac{1}{2}\beta\Re \int_{\R} u_{1}\partial_{x}\left(\abs{u_{2}}^2\overline{u_{1}}\right)\chi\left(\frac{x}{L}\right)dx.
  \end{multline*}
  As for the $A$ term, integrating by parts (several times if necessary) in each term finally leads to: 
  \begin{multline}
    \label{finalB}
    B=-\frac{3}{4L}\int_{\R}\abs{\partial_{x}u_{1}}^2\chi'\left(\frac{x}{L}\right)dx+\frac{1}{4L^3}\int_{\R}\abs{u_{1}}^2\chi'''\left(\frac{x}{L}\right)dx
    \\+\frac{3}{8L}\mu_{1}\int_{\R} \abs{u_{1}}^4\chi'\left(\frac{x}{L}\right)dx
    \\+\frac{1}{2}\beta\Re \int_{\R} \abs{u_{2}}^2u_{1}\partial_{x}\overline{u_{1}}\chi\left(\frac{x}{L}\right)dx +\frac{1}{2L}\beta\Re\int_{\R} \abs{u_{2}}^2\abs{u_{1}}^2 \chi'\left(\frac{x}{L}\right)dx.
  \end{multline}
  Combining~\eqref{finalA} and~\eqref{finalB}, and using the fact that
  $\chi'$ and $\chi'''$ are supported on $[-L,L]$ lead to the following estimate: 
  \begin{equation}
    \label{eq:36}
    \abs*{\frac{d}{dt}\calPlocu(t)}\leq \frac{C}{L}\int_{-L}^L\left(\abs{\partial_{x}u_{1}}^2+\abs{u_{1}}^2+\abs{u_{1}}^4\right)dx + C \int_{\R}\abs{u_{2}}^2\left(\abs{u_{1}}^2+\abs{\partial_{x}u_{1}}^2\right) \chi_L^{1}(x) dx.
  \end{equation}
  Note that for $x\in[-L,L]$, $T_0$ large enough, and using~\eqref{eq:estmod} we have (see Lemma
~\ref{localisation-estimates} for similar arguments)
  \[
  |\partial_xR_j(t,x)+R_j(t,x)|\leq Ce^{-\sqrt{\tilde\omega_j}|x-\tilde x_j|}\leq
  Ce^{-\frac12\sqrt{\tilde\omega_j}|\tilde x_j|}\leq Ce^{-3\sqrt{\omega_*}v_*t}.
  \]
  Note that $C$ here depends on $v_j$.
  Expanding now $u_j=\tilde R_j+\eps_j$, in~\eqref{eq:36}, using the
  above estimate and Lemma~\ref{localisation-estimates}, we obtain
  \begin{equation*}
    \abs*{\frac{d}{dt}\calPlocu(t)}\leq
    \frac{C}{L}\norm*{\binom{\eps_1}{\eps_2}}_{H^1\times H^1}^2+\calO(e^{-3\sqrt{\omega_*}v_*t}).
  \end{equation*}
  The result
  follows integrating in time between $t$ and $T^{n}$ and using the
  bootstrap assumption~\eqref{eq:BH1-bis}. 
\end{proof}

\subsection{Control of the modulation parameters}

We now give an estimate of the variations of $\tildo_{1}$ with respect
to time. This estimate is better than~\eqref{eq:estmod} given by the modulation lemma. 

\begin{lemma}[Variations of $\tildo_{j}(t)$]
  \label{variation-omega1}  
  For $j=1,2$ and for all $t\in [t_{0},T^n]$, we have
  \[
  \abs{\tildo_{j}(t)-\omega_j}\leq
  C\norm{\eps_{j}(t)}_{L^2}^2.
  \]
\end{lemma}
\begin{proof}
  This estimate is due to the choice of the modulation orthogonality
  condition $\psld{\eps_{j}(t)}{\tildR_{j}(t)}=0$, and the conservation of
  the mass of each component:
  \begin{multline*}
    0=\norm{u_j}_2^2-\norm{R_j}_2^2=\norm{\tilde
      R_j}_2^2-\norm{R_j}_2^2+\norm{\eps_j}_2^2\\
    =(\tilde\omega_j-\omega_j)\frac{\partial}{\partial\omega}_{|\omega=\omega_j}\norm{Q_\omega}_2^2+\calO(|\tilde\omega_j-\omega_j|^2)+\norm{\eps_j}_2^2.
  \end{multline*}
  Since
  $\frac{\partial}{\partial\omega}_{|\omega=\omega_j}\norm{Q_\omega}_2^2<0$,
  this concludes the proof.
\end{proof}

\subsection{Conclusion}

With the elements of the previous subsections in hand, we can now
conclude the proof of Proposition~\ref{bootstrap}.

\begin{proof}[End of the proof of Proposition~\ref{bootstrap}]
  Recall that we have made the bootstrap assumption~\eqref{eq:BH1-bis}
  and that our goal is to prove that for all $t\in[t_0,T^n]$ we have in fact the better
  estimate
  \begin{equation}
    \label{eq:37}
    \norm*{
      \begin{pmatrix}u_{1}^n\\u_{2}^n\end{pmatrix}
      -\begin{pmatrix}R_{1}\\R_{2}\end{pmatrix}
    }_{H^1\times H^1}
    \leq  \frac{1}{2} e^{-\sqrt{\omega_{*}}v_{*}t}.
  \end{equation}
  Let us first expand
  \begin{multline*}
    \norm*{      
      \begin{pmatrix}u_{1}^n(t)\\u_{2}^n(t)\end{pmatrix}
      -\begin{pmatrix}R_{1}(t)\\R_{2}(t)\end{pmatrix}
    }_{H^1\times H^1}
    \\
    \leq 
    \norm*{
      \begin{pmatrix}u_{1}^n(t)\\u_{2}^n(t)\end{pmatrix}
      -\begin{pmatrix}\tilde R_{1}(t)\\\tilde R_{2}(t)\end{pmatrix}
    }_{H^1\times H^1}
    +
    \norm*{
      \begin{pmatrix}\tilde R_{1}(t)\\\tilde R_{2}^n(t)\end{pmatrix}
      -\begin{pmatrix}R_{1}(t)\\R_{2}(t)\end{pmatrix}
    }_{H^1\times H^1}
    \\
    \leq 
    \norm*{
      \begin{pmatrix} \eps_{1}(t)\\\eps_{2}(t)\end{pmatrix}
    }_{H^1\times H^1}
    +C\sum_{j=1,2}|\tilde\omega_j-\omega_j|+\sum_{j=1,2}\calO(|\tilde\omega_j-\omega_j|^2).
  \end{multline*}
  By Lemma~\ref{variation-omega1}, the part involving
  $|\tilde\omega_j-\omega_j|$ is controlled by the $\eps$-part. Hence to
  finish the proof it is sufficient to control $\eps$.
  Using successively the coercivity of $\mathcal H$ given in~\eqref{coercivity},
  the expansion~\eqref{tildS:expansion} of the global action $
  \mathcal S$, the conservation of energy and mass, and the almost
  conservation of localized momenta~\eqref{cons-momentum}, we obtain
  \begin{multline*}
    \norm*{
      \begin{pmatrix} \eps_{1}(t)\\\eps_{2}(t)\end{pmatrix}
    }_{H^1\times H^1}^2
    \leq 
    C \mathcal H\begin{pmatrix} \eps_{1}(t)\\
      \eps_{2}(t) \end{pmatrix}
    \\
    \leq 
    \mathcal S
    \begin{pmatrix} u_1(t)\\u_2(t)\end{pmatrix}-
    \mathcal S
    \begin{pmatrix} \tilde R_1(t)\\\tilde R_2(t)\end{pmatrix}
    +\calO(e^{-3\sqrt{\omega_*}v_*t})\leq \frac{C}{L}e^{-2\sqrt{\omega_*}v_*t}.
  \end{multline*}
  Therefore, choosing $L$ large enough we obtain the required estimate
~\eqref{eq:37}. This concludes the proof of Proposition~\ref{bootstrap}.
\end{proof}

\section{Numerical schemes}
\label{sec:scheme}

We describe here the numerical methods that we will be using in the
next section.

\subsection{The time-splitting spectral method}
We start by
the time-splitting spectral method that we use to
solve~\eqref{eq:nls} numerically. The equations are solved on a bounded interval $I=(-a,a)$.
We use a uniform spatial grid with mesh size $h>0$ and grid points
$x_k=x_0+k h$, $k=0,\ldots,K$, where $K+1\in\N$ is the (odd) number of grid points. 
Then $h=2a/K$. The time grid is given by
$t_n=t_0+n\tau$, $n\in\N_0$, where $\tau>0$ is the time step size and $t_0$ the initial time. We set
$(u_j)_k^n:=u_j(t_n,x_k)$, where $j=1,2$, $k=0,\dots,K$, and $n\in\N_0$.
We split the system~\eqref{eq:nls} into the two subsystems
\begin{align}
  i\partial_tu_j &=  -\mu_{j}|u_j|^2u_j
  - \beta |u_{3-j}|^2u_j, \quad j=1,2,  \label{eqpotn1} \\
  i\partial_tu_j &= -\frac12\partial_{xx} u_j,  \quad j=1,2, \label{eqdelta} 
\end{align}
considered on $[t_n,t_{n+1}]$ and subject to some initial data.
These subsystems are solved as follows.

\begin{description}
\item[\rm\em Step 1] Computing the evolution of 
~\eqref{eqpotn1} we observe that the quantities $|u_j|^2$ remain unchanged. Therefore, we ``freeze''
  these values at time $t_n$ and solve the resulting linear ODEs exactly
  in the interval $[t_n,t_n+\tau/2]$, giving at time $t_n+\tau/2$:
  \begin{equation*}
    (u_1)_k^* = \exp\left(i\frac{\tau}{2}\Big(
      \mu_{1}|(u_1)_k^n|^2 + \beta|(u_2)_k^n|^2\Big)\right)
    (u_1)_k^n,
  \end{equation*}
  and analogously for $(u_2)_k^*$. 
\item[\rm\em Step 2] We solve~\eqref{eqdelta} for $j=1,2$ in the 
  interval $[t_n,t_{n}+\tau]$,
  discretized in space by the Fourier spectral method and solved exactly in time:
  \begin{equation*}
    (u_j)_k^{**} = \frac{1}{K+1}\sum_{m=-K/2}^{K/2}\exp\left(-i\tau\,
      \nu_m^2\right)(\widehat u_j)_m^*\exp\big(i\nu_m(x_k-x_0)\big),
    \quad j=1,2,
  \end{equation*}
  where $\nu_m = 2\pi m/(x_K-x_0)$ and
  \[
  (\widehat u_j)_m^* = \sum_{l=0}^{K}(u_j)_l^*\exp\left(-i\nu_m(x_l-x_0)\right),
  \quad m=-\frac{K}{2},\ldots,\frac{K}{2}.
  \]
\item[\rm\em Step 3] We solve~\eqref{eqpotn1} on $[t_n+\tau/2,t_{n+1}]$ 
  using the discretization of Step 1 with $(u_j)^{**}_k$ instead of $(u_j)^n_k$
  and obtain $(u_j)_k^{n+1}$.
\end{description}

\subsection{The normalized gradient flow}
We will also need to compute the ground state solution of the
following elliptic system with fixed masses

\begin{equation}
  \label{eq:ellsys}
  \left\{
    \begin{aligned}
      -\partial_{xx} \phi_1 + \omega_1 \phi_1 -\mu_1 \phi_1^3 -\beta \phi_2^2 \phi_1=0,\\
      -\partial_{xx} \phi_2 + \omega_2 \phi_2 -\mu_2 \phi_2^3 -\beta \phi_1^2 \phi_2=0.
    \end{aligned}
  \right.
\end{equation}
To that purpose, we use the normalized gradient flow.
The problem can also be viewed as a nonlinear eigenvalue problem with $\omega_1,\omega_2$ being the eigenvalues, which can be computed from the corresponding eigenfunctions ($j=1,2$):
\begin{equation*}
  \omega_j^{\phi}=\frac{\displaystyle\int_{\R}\left( -|\partial_{x} \phi_j|^2+\mu_j \phi_j^4+ \beta \phi_1^2 \phi_2^2\right)dx}{ \displaystyle\int_{\R}\phi_j^2 dx}.
\end{equation*}
We solve~\eqref{eq:ellsys} by normalized gradient flow with given $(a_1,a_2)$, such that: 
\begin{equation*}
 \int_{\R}\phi_1^2 dx=a_1^2 \quad \mbox{ and } \quad \int_{\R}\phi_2^2 dx=a_2^2.
\end{equation*}
The standard gradient flow with discrete normalization consists in introducing an imaginary time in the nonlinear Schr\"odinger equations, thus looking at the imaginary time propagation ($t\rightarrow- it$) and after every step project the solutions such that the $L^2$-norms are equal to $(a_1^2,a_2^2)$.
In~\cite{BaoDu03} the authors present the normalized gradient flow,
prove it is energy diminishing, and propose numerical methods to discretize it. Hereafter we adapted the normalized gradient flow for the given system and discretized it by a semi-implicit Backward Euler finite differences scheme.
For $(\phi_j)_{k}^n=\phi_j(t^n,x_k)$ being the discrete solution, $x_k=x_0+k \cdot h$ the grid points with $k=0,1,\dots,K-1,K$ and the time sequence $0<t_1<t_2<\dots<t_n<t_{n+1}<\dots$ with $\tau=t_{n+1}-t_n$ we have the following discretization of the normalized gradient flow:
\begin{description}
\item[\rm\em Step 1] 
  We first solve on $[t_n,t_{n+1}]$ with initial data $\phi_j(t_n,x_k)$:
  \begin{eqnarray*}
    \frac{(\phi_j)_{k}^{*}-(\phi_j)_{k}^{n}}{\tau} = \frac{(\phi_j)_{k+1}^{*}-2(\phi_j)_{k}^{*}+(\phi_j)_{k-1}^{*}}{h^2}+\mu_j |(\phi_j)_{k}^{n}|^2(\phi_j)_{k}^{*}+\beta  |(\phi_{3-j})_{k}^{n}|^2(\phi_j)_{k}^{*}
  \end{eqnarray*}
  with $j=\{1,2\}$.
  As solution we get $(\phi_j)_k^*$.

\item[\rm\em Step 2]
  $(\phi_j)_k^*$ is then normalized to get finally $(\phi_j)_k^{n+1}$
  \begin{equation*}
    (\phi_j)_k^{n+1}=\frac{a_k (\phi_j)_k^*}{\norm{ (\phi_j)_k^*}_{L^2}}.
  \end {equation*}
\end{description}
For $t\rightarrow +\infty$ we obtain the ground state solution
$(\phi_1,\phi_2)$ of~\eqref{eq:ellsys} with $L^2$-norms equal to
$(a_1^2,a_2^2)$ and frequency parameters $(\omega_1^{\phi_1},\omega_2^{\phi_2})$.

\section{Numerical experiments}
\label{sec:experiments}

Our ansatz for the initial data in the next experiments is the following. 
\begin{equation}\label{eq:initsol}
  u_j(t_0,x)=e^{i (\omega_j t_0 -v_j^2 t/4+ v_j x/2)} Q_{\omega_j}(x-v_j t_0-x_j) \quad j=1,2
\end{equation}
with $Q_{\omega}$  defined in~\eqref{eq:snls-basic-intro}. 
Without loss of generality, we may assume that $0<v_1=-v_2$ (Galilean
invariance), hence the soliton on the first component will be
traveling to the right and the soliton on the second component will
be traveling to the left. We will also assume that 
$x_1=x_2=0$ (invariance by translation in time and space) and to
guarantee that our solitons on the first and second components are positioned on the left and on the
right respectively and at a sufficiently large distance, we will chose
the initial time to be $t_0=-10$.

\subsection{Purely elastic interaction}
The integrable case has been studied in depth in the book~\cite{AbPrTr04}. In particular, in that case the system is completely integrable via the inverse scattering transform and explicit solutions may be exhibited. For example, there exists a solution $U=(u_1,u_2)$ of~\eqref{eq:nls} which has the following behavior. At time $t\to-\infty$, 
\begin{align}\label{eq:explicit-minus-infty}
  u_1(t,x)&\sim e^{i\left(\frac{1}{2}v_1x+\left(\omega_1-\frac{1}{4}|v_1|^2\right) t\right)}
  \sqrt{\omega_1}Q\left(\sqrt{\omega_1}\left(
      x-v_1t
    \right)\right),
  \\
  \nonumber u_2(t,x)&\sim e^{i\left(\frac{1}{2}v_2x+\left(\omega_2-\frac{1}{4}|v_2|^2\right) t\right)}
  \sqrt{\omega_2}Q\left(\sqrt{\omega_2}\left(
      x-v_2t
    \right)\right),
\end{align}
whereas at time $t\to+\infty$ the two components have interacted and the outcome are solitons of same speed and frequency but with a shift in phase and translation:
\begin{align}\label{eq:explicit-plus-infty}
  u_1(t,x)&\sim \hat{\phi}_1(t,x):=e^{i\theta_1}e^{i\left(\frac{1}{2}v_1x+\left(\omega_1-\frac{1}{4}|v_1|^2\right) t\right)}
  \sqrt{\omega_1}Q\left(\sqrt{\omega_1}\left(
      x-v_1t-\tau_1
    \right)\right),
  \\
  \nonumber u_2(t,x)&\sim \hat{\phi}_2(t,x):= e^{i\theta_2}e^{i\left(\frac{1}{2}v_2x+\left(\omega_2-\frac{1}{4}|v_2|^2\right) t\right)}
  \sqrt{\omega_2}Q\left(\sqrt{\omega_2}\left(
      x-v_2t-\tau_2
    \right)\right),
\end{align}
where the shift parameters are given by the formulas 
\begin{align*}
  \tau_1=-\frac{\ln(|\chi_1|)}{\sqrt{\omega_1}},\quad
  \tau_2=\frac{\ln(|\chi_2|)}{\sqrt{\omega_2}},\quad
  \theta_1=\frac{\chi_1}{|\chi_1|},\quad 
  \theta_2=\frac{\chi_2}{|\chi_2|},
  \\
  \chi_1=\frac{v_1-v_2+i2(\sqrt{\omega_1}+\sqrt{\omega_2})}{v_1-v_2+i2(\sqrt{\omega_1}-\sqrt{\omega_2})},\quad
  \chi_2=\frac{v_1-v_2+i2(\sqrt{\omega_1}+\sqrt{\omega_2})}{v_1-v_2-i2(\sqrt{\omega_1}-\sqrt{\omega_2})}.
\end{align*}
Note that since the speeds and amplitudes of the solitons are not modified the interaction is considered elastic. Note also that we have a pure two-speeds solitary wave at both ends of the time line, without any appearance of dispersion despite the interaction. This is a characteristic feature of completely integrable systems.

The parameters are chosen as follows: $\mu_1=\mu_2=\beta=1$, $\omega_1=5$, $\omega_2=1$, $v_1=1$, $v_2=-1$, and $x_1=x_2=0$, thus we start with two solitons located at $\pm10$, respectively, which move to each other and observe them until $t_{final}=10$.  Furthermore, the integration domain is $I=(-20,20)$ with $K=1024$ spatial grid points and $\tau=10^{-3}$.

\begin{figure}[!htbp]
  \includegraphics[width=60mm]{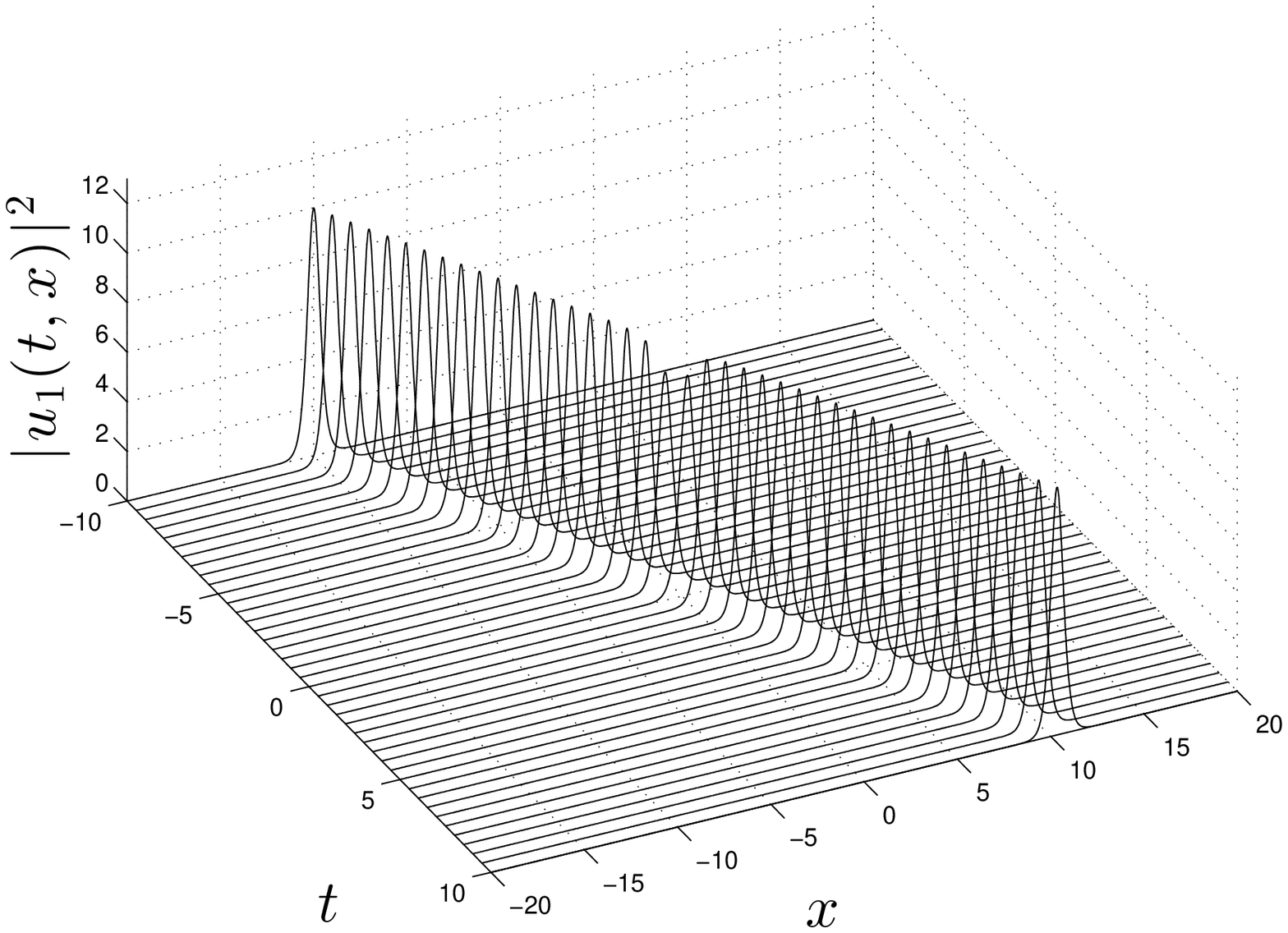}
  \includegraphics[width=60mm]{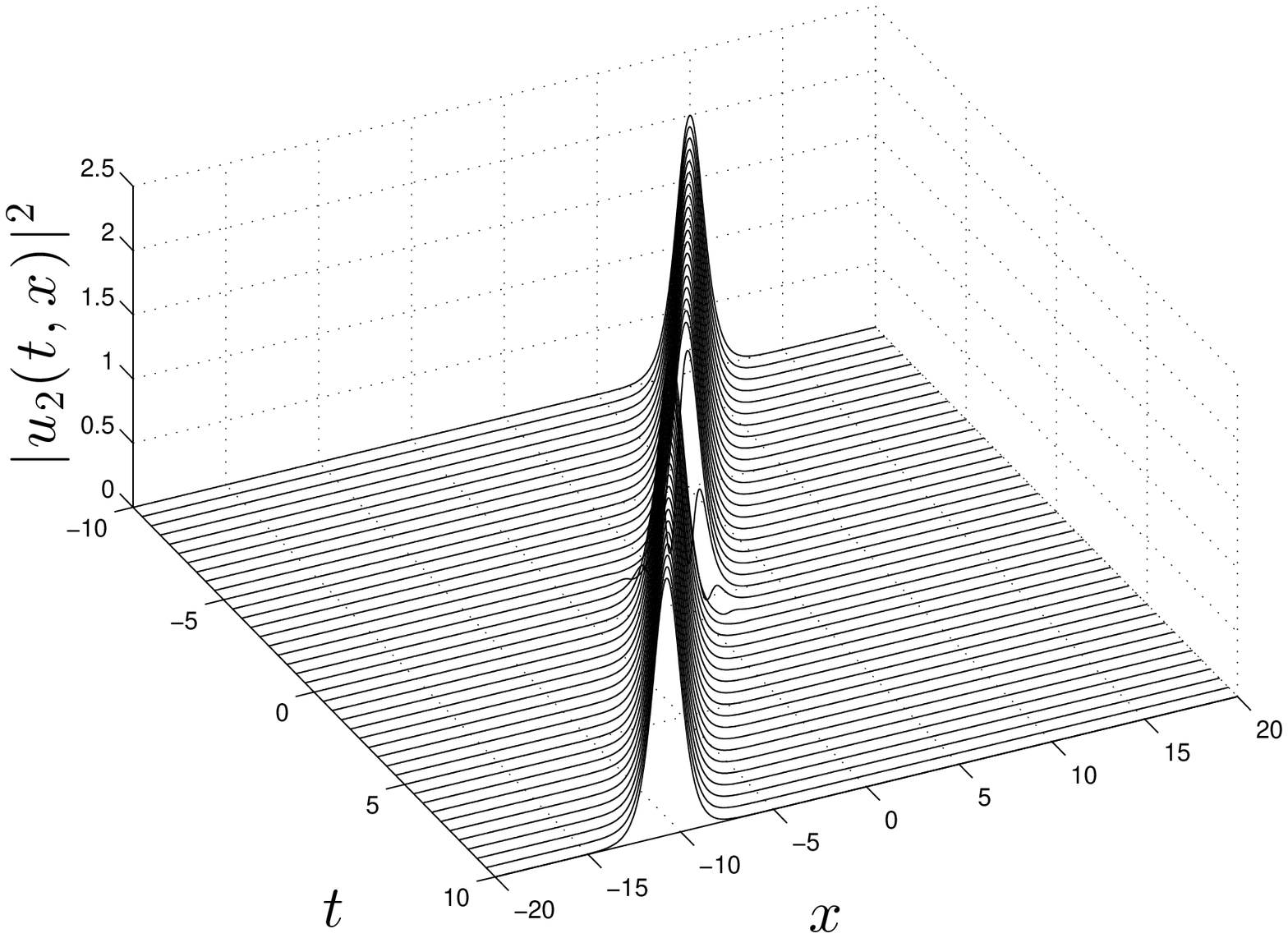}
  \caption{{\bf Purely elastic interaction:} Position densities $|u_1(t,x)|^2$ and $|u_2(t,x)|^2$ as functions of space and time for $\mu_1=\mu_2=1$ and $\beta=1$}
  \label{fig:ex1a}
\end{figure}

We plot the solution $(u_1,u_2)$ to~\eqref{eq:nls} and observe that
the two solitons remain unchanged after the interaction, excepting a
shift in phase and translation. The shift in translation can be
observed in Figure~\ref{fig:ex1a}. Our numerical experiments are in
good line with the theoretically predicted behavior. We have compared
pointwise the numerical solution $|u_1|^2$ obtained by taking
\eqref{eq:explicit-minus-infty} as initial data at $t=-10$ and the
theoretical outcome $|\hat{\phi}_1|^2$ given by
\eqref{eq:explicit-plus-infty}. The results are shown in Figure~\ref{fig:ex1b}.

\begin{figure}[!htbp]
  a) \includegraphics[width=55mm]{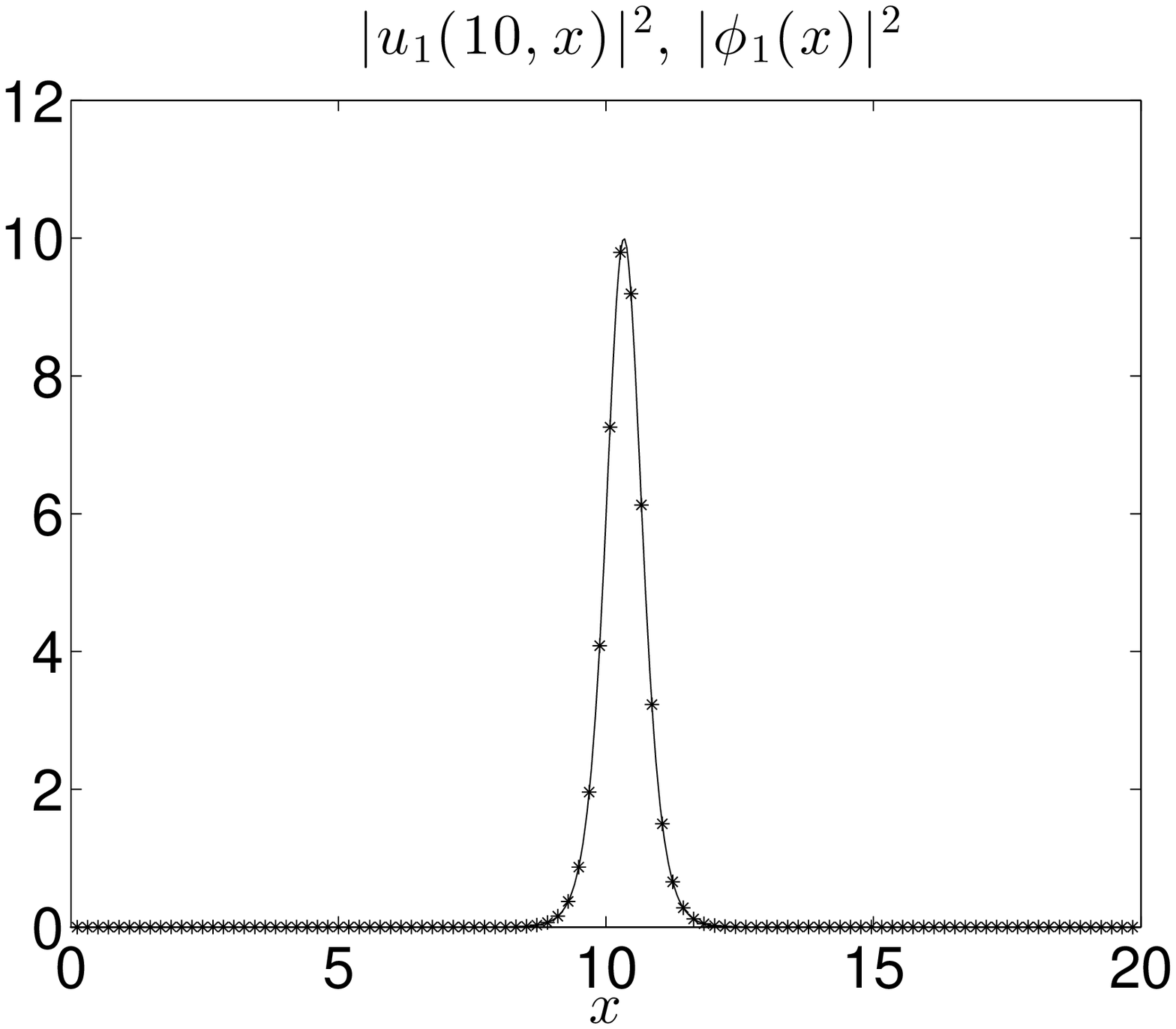}
  b) \includegraphics[width=55mm]{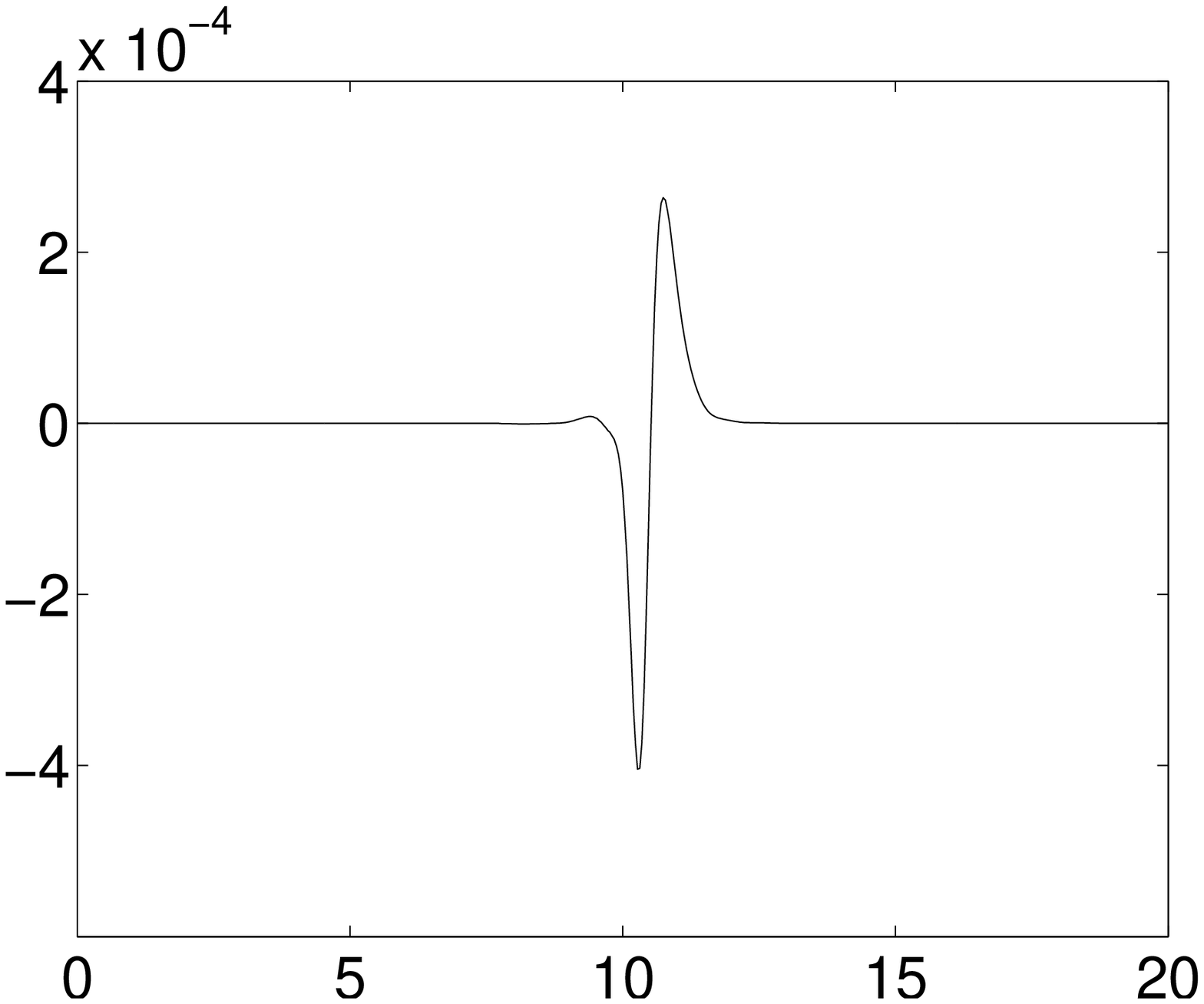}
  \caption{a) Plots of $|u_1(10,x)|^2$ ('-' line) and $|\hat{\phi}_1(x)|^2$ ('.' line). 
    b) Difference of $|u_1(10,x)|^2-|\hat{\phi}_1(x)|^2$ }
  \label{fig:ex1b}
\end{figure}

\subsection{Symmetric collision}
In this experiment the outcome are still multi-speed solitary
waves. After the collision we observe a small part of the soliton
moving along with the soliton of the other component. We see that at the final time the solution fits well the ground state solution obtained by minimizing the system with fixed masses.
Here we choose $\mu_1=\mu_2=1$ and $\beta=3$. We have solved the system~\eqref{eq:nls} 
in the interval $(-200,200)$ with periodic boundary conditions, initial time $t_0=-10$ and initial condition
\eqref{eq:initsol}
and the parameters  $\omega_1=\omega_2=1$, $v_1=2$, $v_2=-2$, and $x_1=x_2=0$, thus we start with two solitons located at $\pm10$, respectively, which move to each other and observe them till $t_{final}=40$.   Moreover, we use $K=4096$ in space and time step  $\tau=10^{-3}$.

\begin{figure}[!htbp]
  \includegraphics[width=60mm]{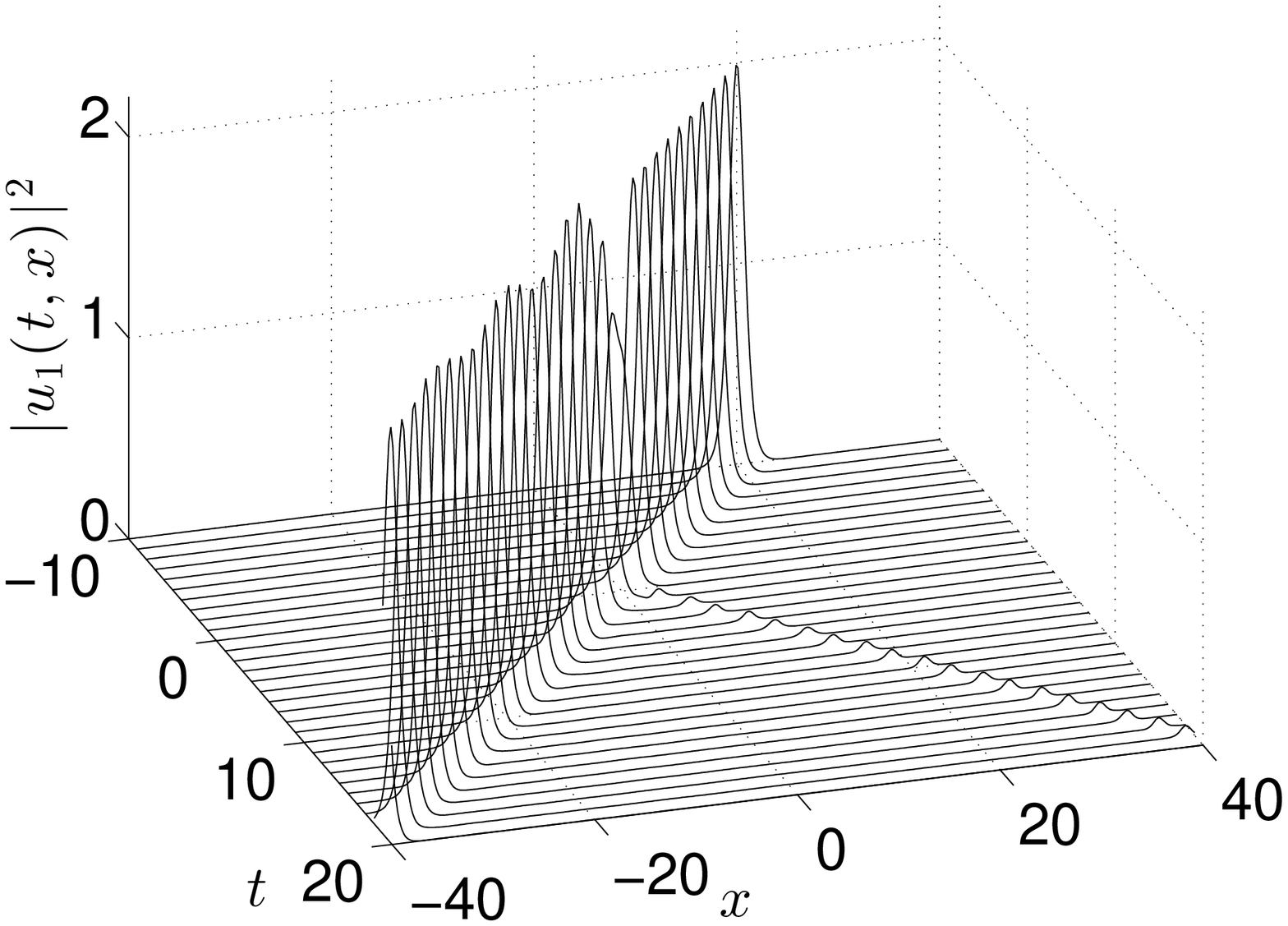}
  \includegraphics[width=60mm]{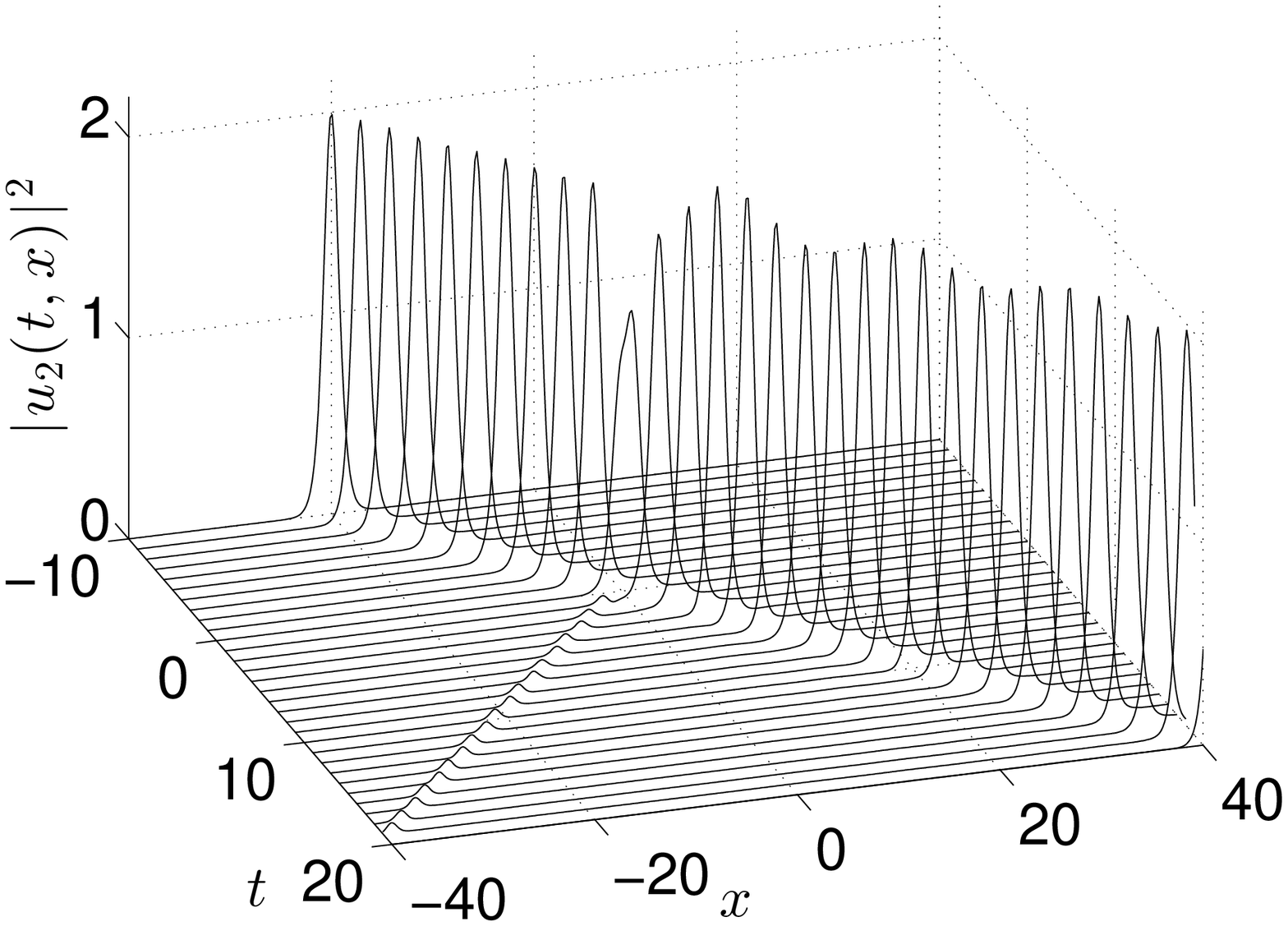}
  \caption{{\bf Symmetric collision:} Position densities $|u_1(t,x)|^2$ and $|u_2(t,x)|^2$ as functions of space and time for $\mu_1=\mu_2=1$ and $\beta=3$}
  \label{fig:ex2a}
\end{figure}

Here the interaction parameter $\beta$ is greater than in the first experiment and after the interaction at time zero we see in each component a small part going with the other component, respectively (Figure~\ref{fig:ex2a}).

At the final time $t=40$ we look at the left side part of each
component $u_j^-=u_j \cdot \chi_{[-\infty,0]}$, and compute the
$L^2$-norms 
\[
a_j^2=\int_{\R} |u^-_j(40,x)|^2dx.
\]
Note that the mass of each component  $M_j(t)=\frac12 \int_{\R}
|u_j(t,x)|^2dx$ is conserved, $M_j(0)=M_j(t)$ for $j=1,2$ and  for all $t\in[-10,40]$.
At $t=40$ we obtain $(a_1=\sqrt{3.893},a_2=\sqrt{0.069})$. 

\begin{figure}[!htbp]
  a)\includegraphics[width=59mm]{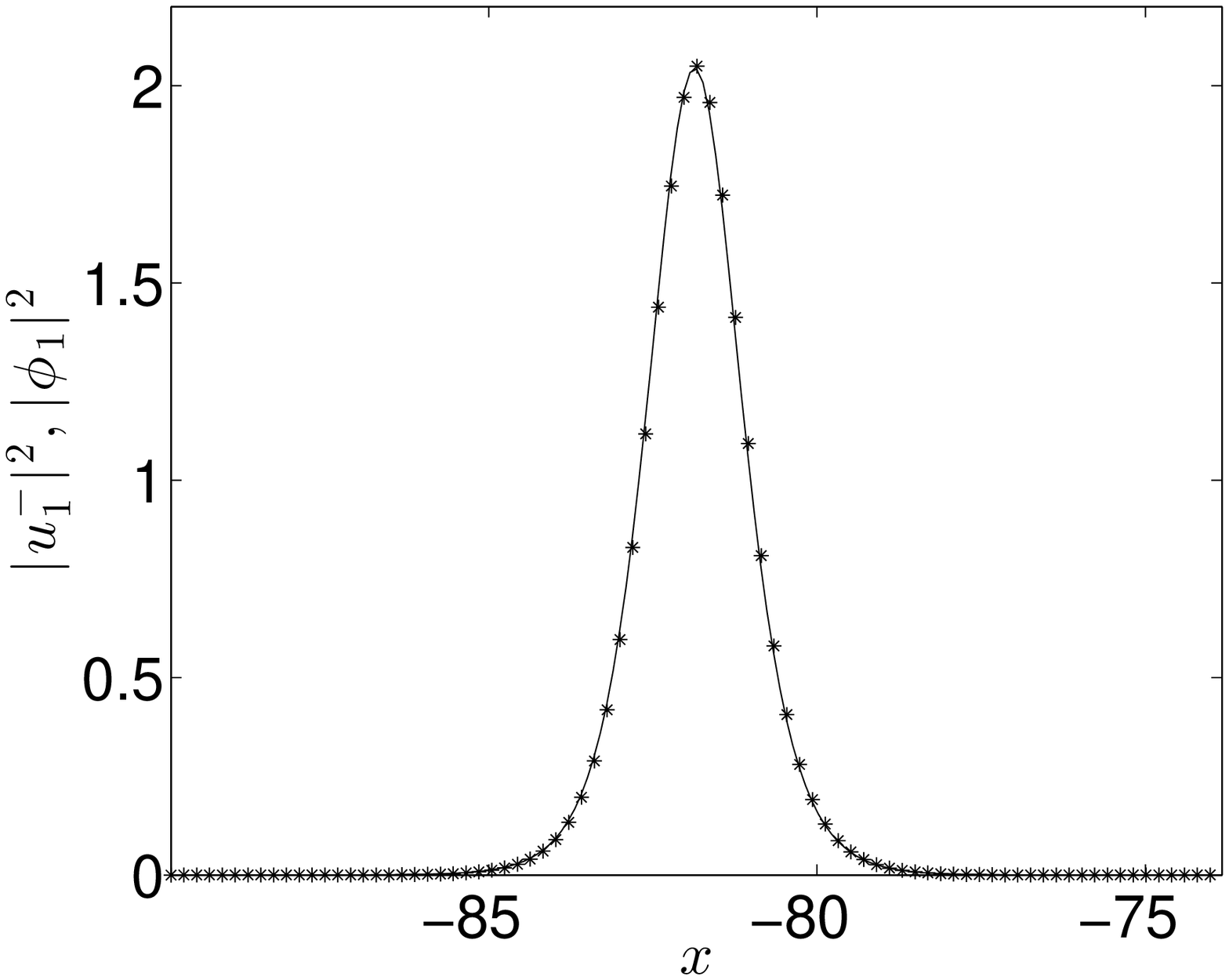}
  b)\includegraphics[width=59mm]{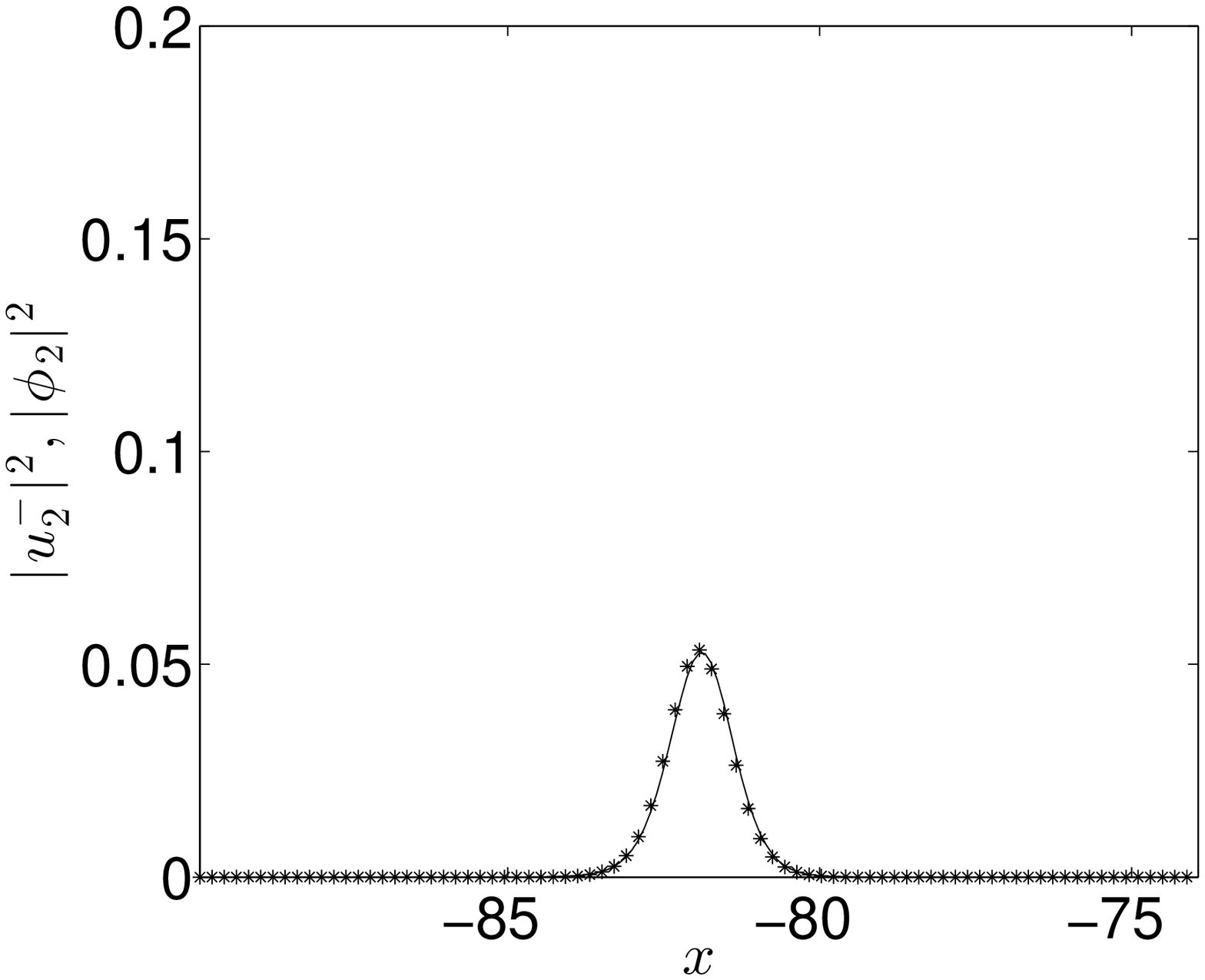}
  \caption{{\bf Symmetric collision:} a) Plots of $|u_1^-(40,x)|^2$ ('-' line) and $|\phi_1(x+x_1)|^2$ ('.' line)  b) Plots of $|u_2^-(40,x)|^2$ and $|\phi_2(x+x_2)|^2$, 
    where $x_j$ is the $x$ corresponding to the $\max{|u_j|^2}$, with $j=1,2$.}
  \label{fig:ex2b}
\end{figure}

Finally we compute the ground state solution $(\phi_1(x),\phi_2(x))$
of the elliptic system~\eqref{eq:ellsys} by the gradient flow with
constraints $(3.893,0.069)$ and compare it to
$(u_1^-(40,x),u_2^-(40,x))$. In Figure~\ref{fig:ex2b} we first shift
$\phi_1(x+x_1)$ (with $x_1$ corresponding to $\max{|u_1|^2}$) and
$\phi_2(x+x_2)$  (with $x_2$ corresponding to $\max{|u_2|^2}$) and
then compare the position density pointwise to the position density of
the final solution located on the left of the origin
$(u_1^-(40,x),u_2^-(40,x))$ and realize that they fit very well.

In this case, the interaction result into a new repartition of the mass and
of the energy so as to approximate a ground state profile. 

\subsection{Dispersive inelastic interaction}
In this experiment we observe a loss of energy, mass and momentum in a small dispersive part located at the interaction place and a small dispersive part moving to the boundaries. Here we have solved the system~\eqref{eq:nls}  with $\beta=-1$  and $\mu_1=\mu_2=1$
in the interval $(-500,500)$ with $8192$ grid points for the spatial discretization, with periodic boundary conditions, initial time $t_0=-10$ and initial condition
\eqref{eq:initsol} with the following parameters $v_1=-v_2=2.7$, $\omega_1=\omega_2=1$ , and $x_1=x_2=0$. 
\begin{figure}[!htbp]
  \includegraphics[width=59mm]{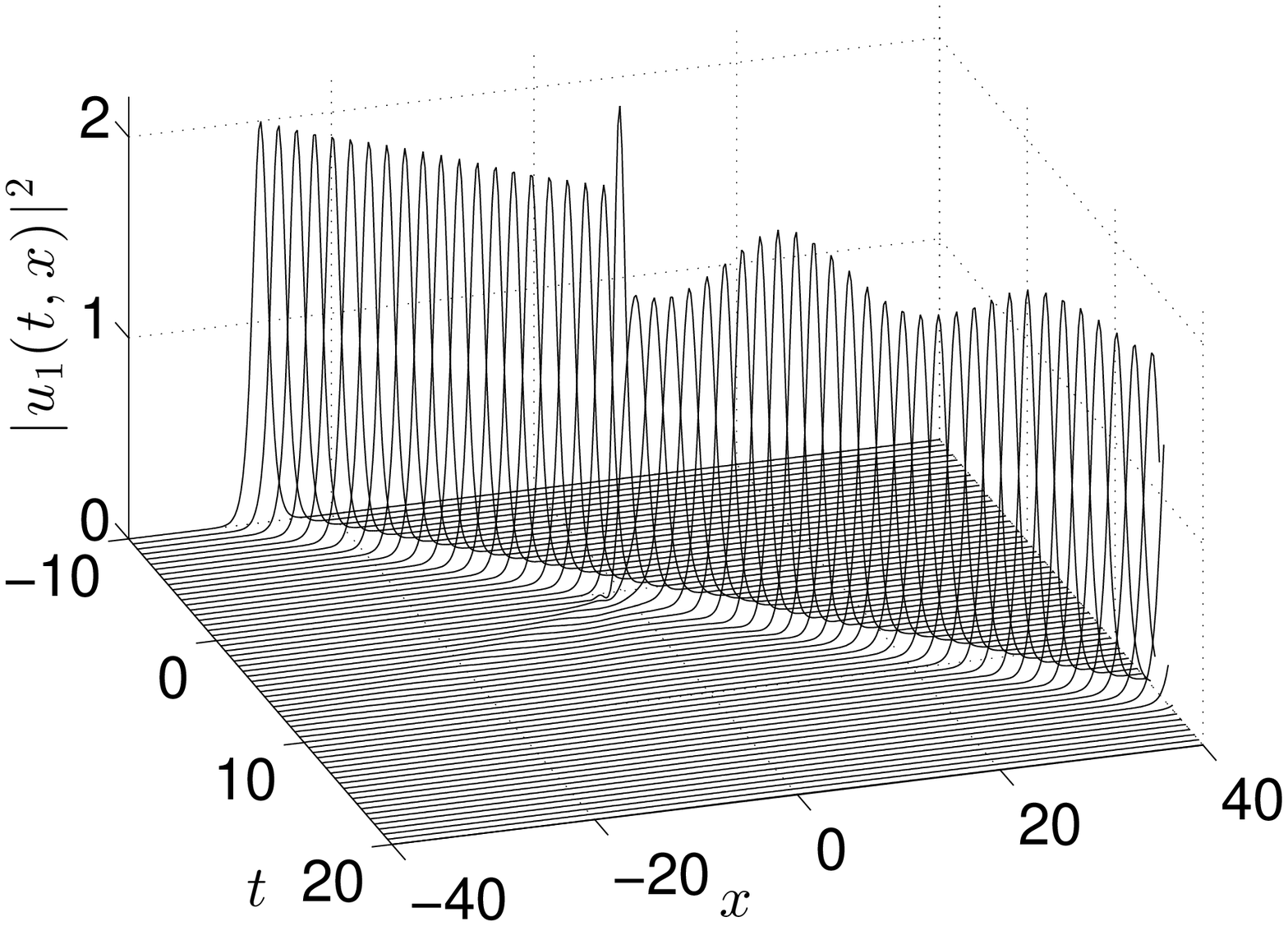}
  \includegraphics[width=59mm]{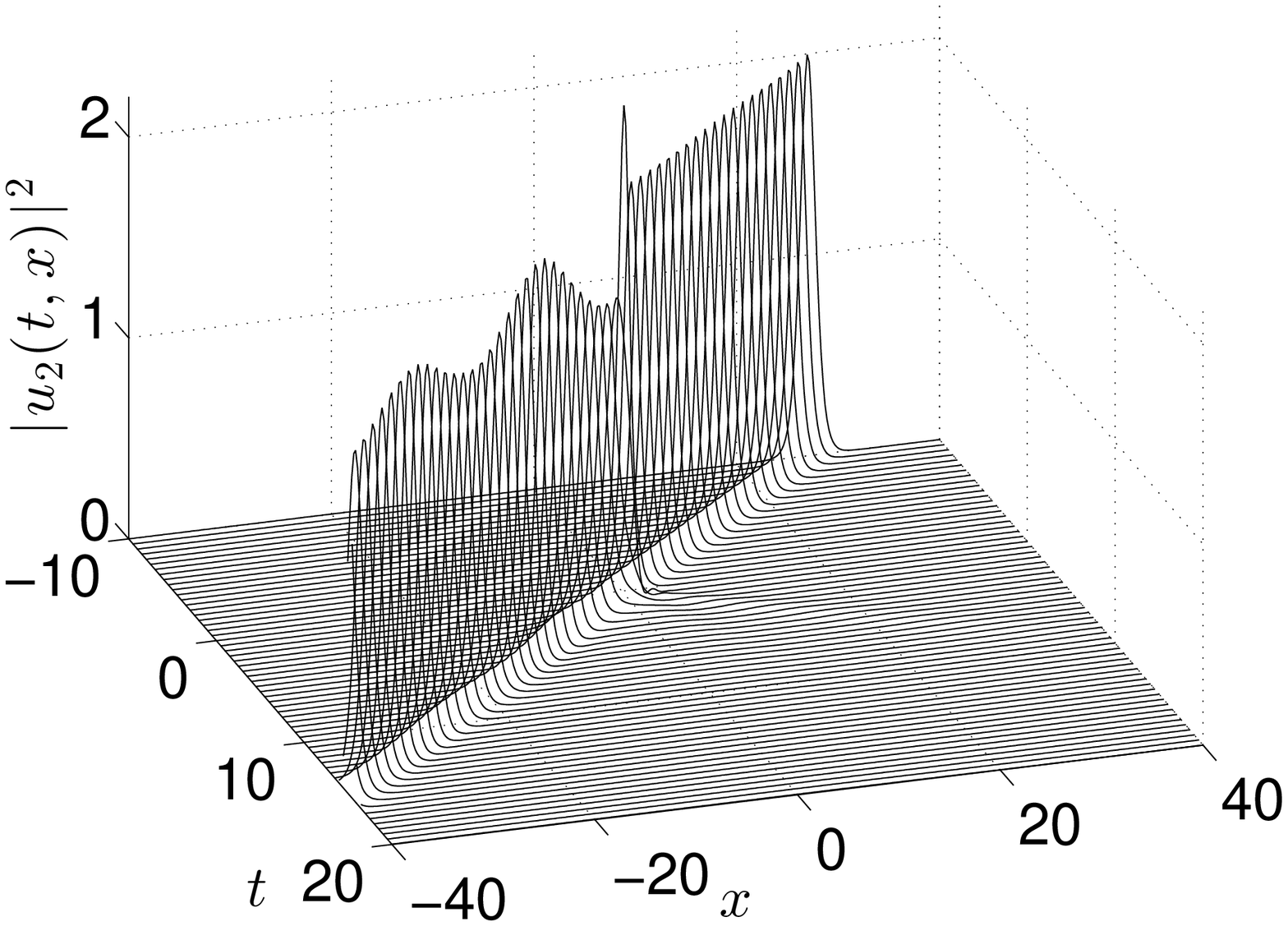}
  \caption{{\bf Dispersive inelastic interaction:} Plots of $|u_1(t,x)|^2$  (left) and $|u_2(t,x)|^2$ (right) as functions of space and time}
  \label{fig:ex3a}
\end{figure}

\begin{figure}[!htbp]
  \includegraphics[width=59mm]{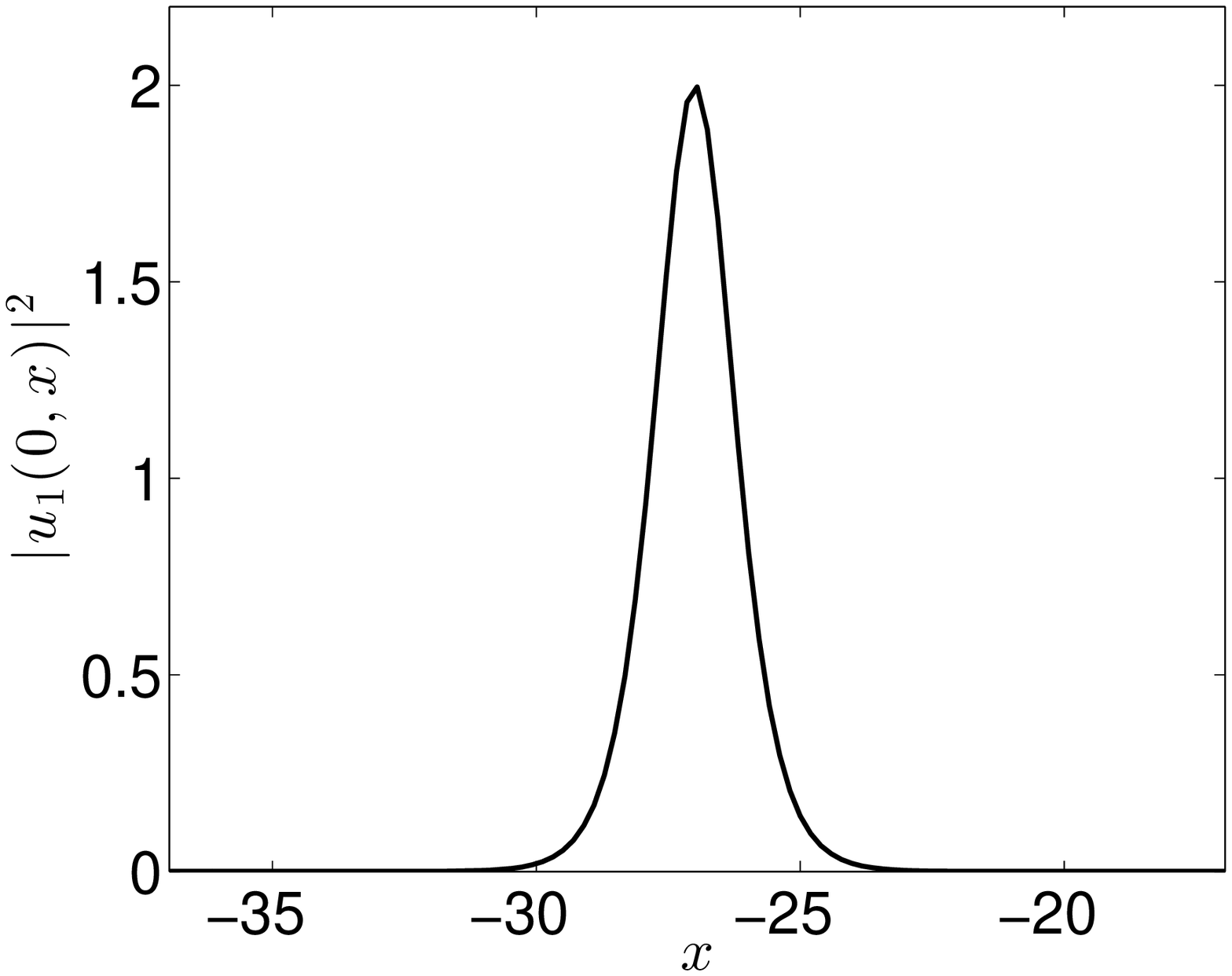}
  \includegraphics[width=59mm]{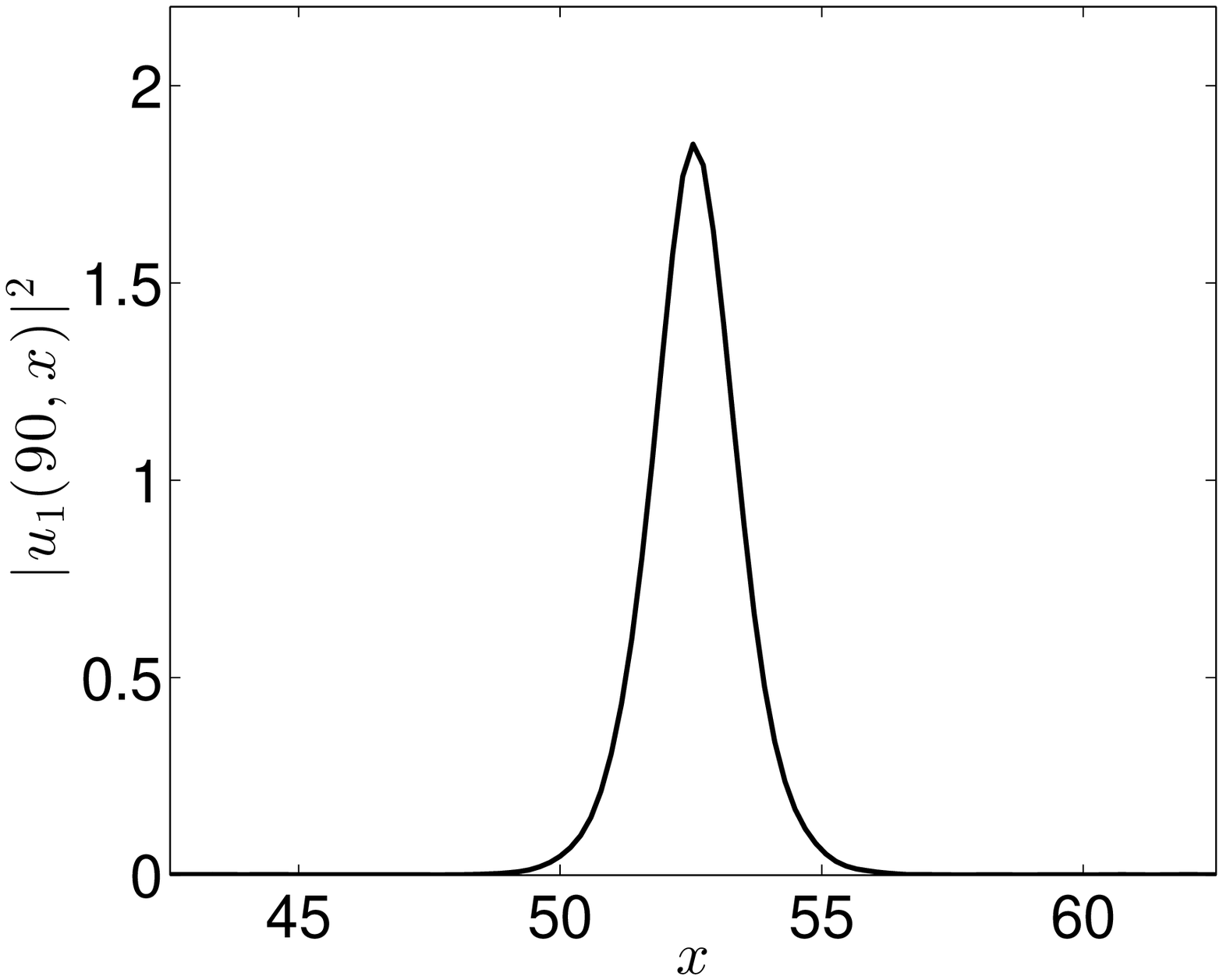}
  \caption{{\bf Dispersive inelastic interaction:} Plots of $|u_1(0,x)|^2$  (left) and $|u_1(90,x)|^2$ (right)}
  \label{fig:ex3b}
\end{figure}

Figure~\ref{fig:ex3a} shows the evolution of the waves.
 In Figure~\ref{fig:ex3b} we see the position density of the initial soliton $u_1(0,x)$ and the one of the final soliton at 
 $t=90$, with lost mass and energy.

\subsection{Reflexion}
In this experiment we observe a change of the sign of the speeds, thus the solitons are reflected after interaction. Here we have the same parameter as in the case of a dispersive interaction, $\beta=-1$  and $\mu_1=\mu_2=1$, but the initial velocity is smaller, precisely we have the initial condition
\eqref{eq:initsol} with the following parameters $v_1=-v_2=0.5$, $\omega_1=\omega_2=1$, and $x_1=x_2=0$. Since the velocity is not so large, and we observe the soliton until $t_{final}=10$, we solve~\eqref{eq:nls} on  $(-20,20)$ with $1024$ spatial grid points, periodic boundary conditions and time step equal to $\tau=10^{-3}$.

\begin{figure}[!htbp]
  \includegraphics[width=59mm]{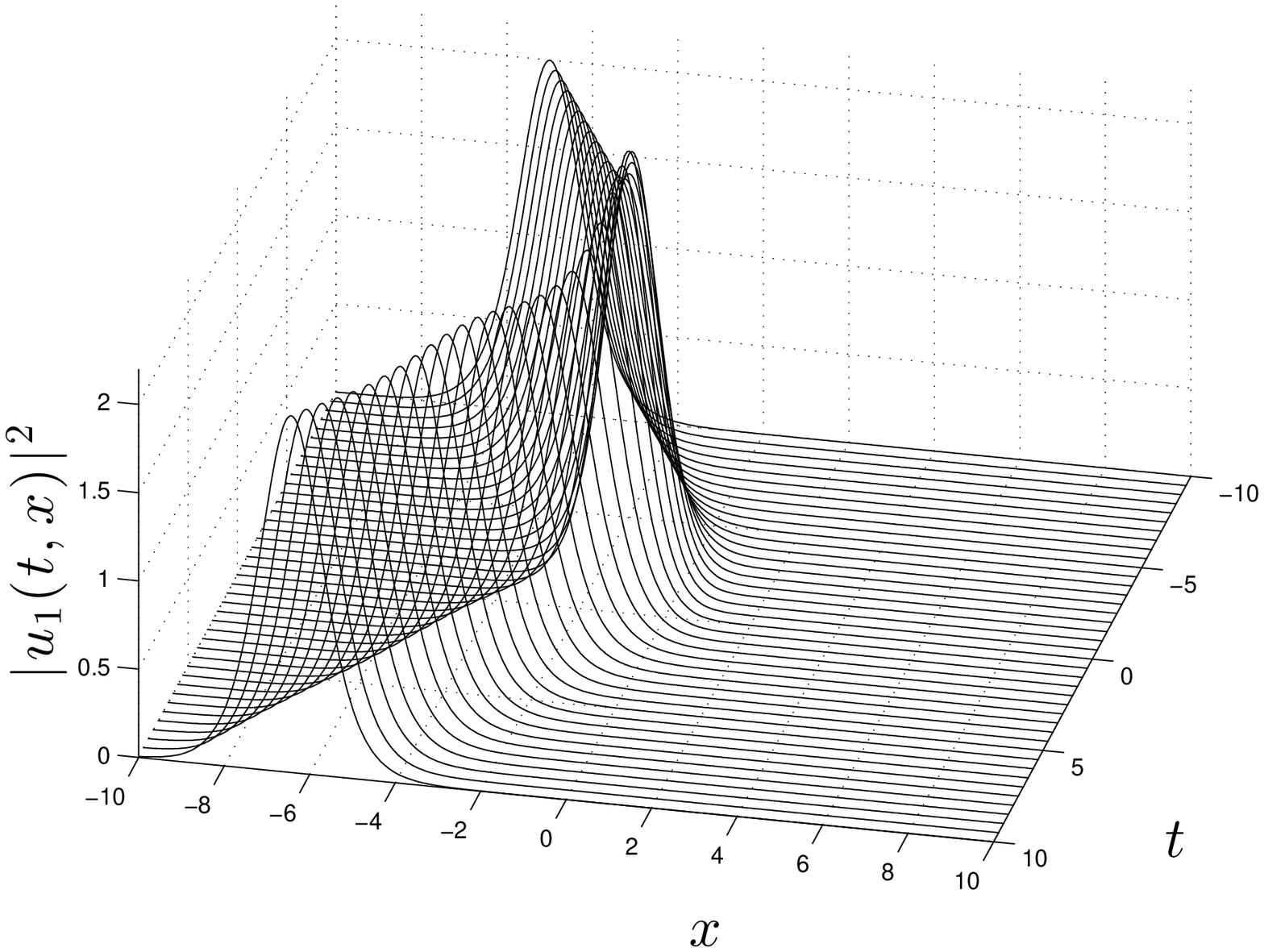}
  \includegraphics[width=59mm]{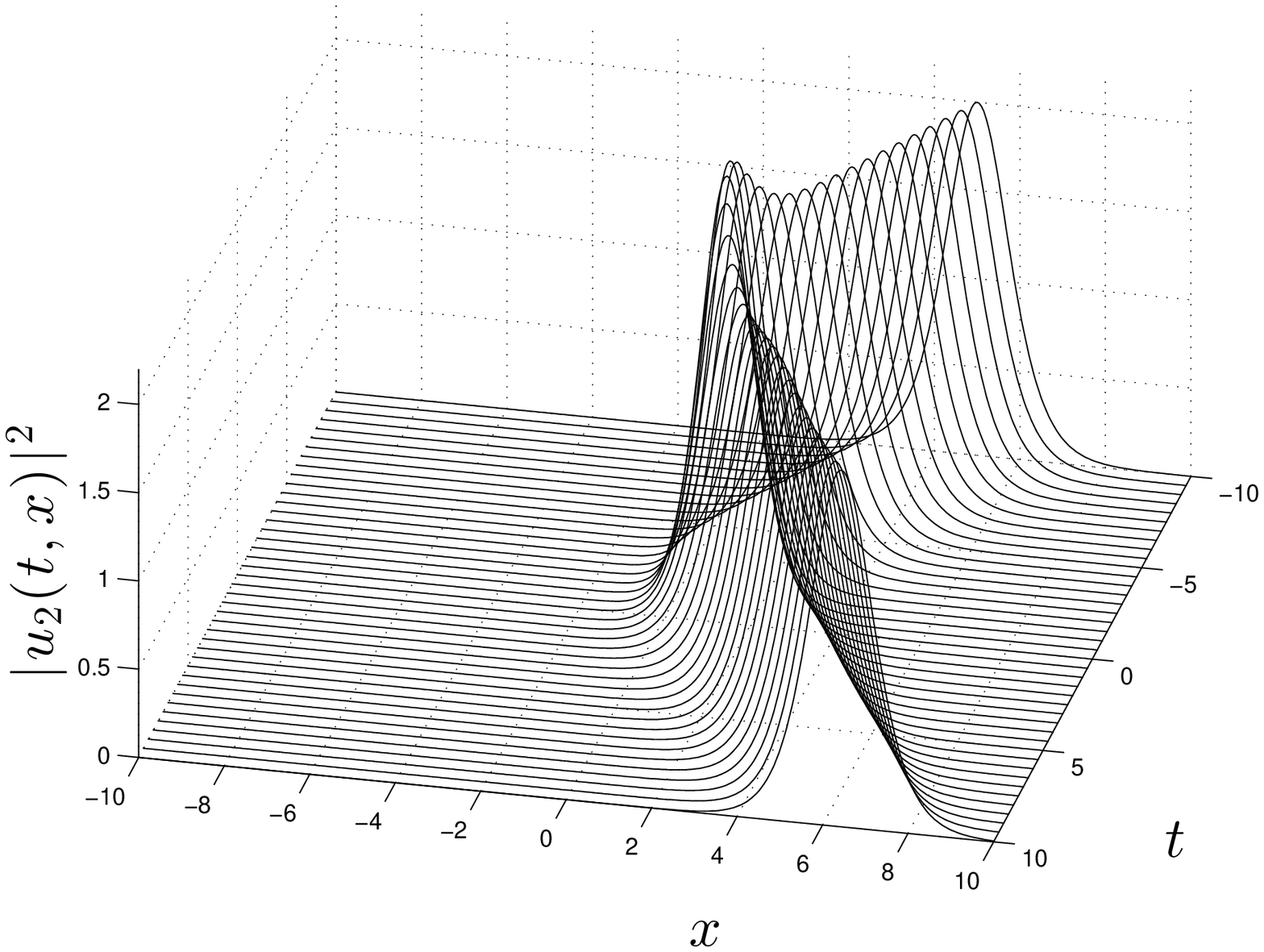}
  \caption{{\bf Reflexion:}  Plots of $|u_1(t,x)|^2$  (left) and $|u_2(t,x)|^2$ (right) as functions of space and time}
  \label{fig:ex4a}
\end{figure}

In Figure~\ref{fig:ex4a}  we see the reflexion of the two solitons after interaction, thus only the velocities changed sign. In Figure~\ref{fig:ex4b} we compare point wise the solution at the final time to the initial solution, and observe that it remains unchanged.
\begin{figure}[!htbp]
  \includegraphics[width=59mm]{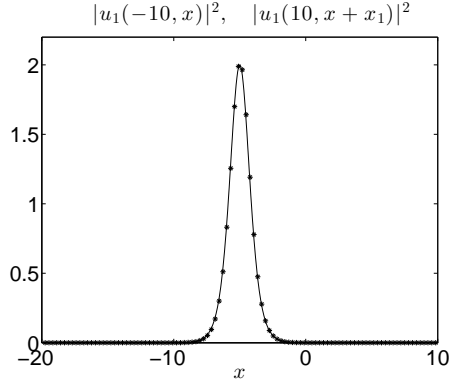}
  \caption{{\bf Reflexion:}  Plots of $|u_1(-10,x-x_1)|^2$  ('$-$' line) and $|u_1(10,x)|^2$ ('.' line), where $x_1$ is the translation such that we can compare the two solitons}
  \label{fig:ex4b}
\end{figure}

These were various examples of the possible outcomes of multi-speeds solitary waves interactions. The above described situations may not be the only possibles and we do not claim comprehensiveness. 
\FloatBarrier

\appendix

\section{Proof of the modulation lemma}
\label{appendix}

\begin{proof}[Proof of Lemma    \ref{lem:modulation}] The proof of the modulation lemma is inspired from~\cite{MaMe06} and relies on the implicit function theorem. 

  Let us define $F: \hu\times \R_{+}\times \R\times \R \longrightarrow \R^{3}$ by 
  \begin{equation*}
    F(u,\om,\gamma,y)=\begin{pmatrix}
      \psld{u-R(\om,\gamma,y)}{R(\om,\gamma,y)}\\
      \psld{u-R(\om,\gamma,y)}{iR(\om,\gamma,y)}\\
      \psld{u-R(\om,\gamma,y)}{\partial_{x} R(\om,\gamma,y)}
    \end{pmatrix}
  \end{equation*}
  where, for the sake of simplicity, we removed the $i$ indexes and $R(\om,\gamma,y)$ denotes the soliton 
  $R(\om,\gamma,y)=e^{i(\frac{1}{2}v\cdot x +\gamma)}\frac{1}{\sqrt{\mu}}Q_{\om}(x-y)$. Note that, 
  \begin{equation*}
    \forall (\om, \gamma,y)\in \R_{+}\times \R\times \R,\quad F(R(\om,\gamma,y),\om,\gamma,y)=0.
  \end{equation*}
  The key idea here is to apply the implicit function theorem to function $F$ at point $(R(\om,\gamma,y),\om,\gamma,y)$ for fixed $(\om, \gamma,y)\in \R_{+}\times \R\times \R$. Simple computations due to properties of $Q_{\om}$, mainly the fact that $Q_{\om}$ is radial and even,  lead to
  $\nabla F(R(\om, \gamma,y),\om, \gamma,y)$  diagonal with diagonal terms $a, b, c$ given by:
  \begin{align*}
    a&=\frac{1}{\mu}\Re \int_{\R}\partial_{\om}Q_{\om}(x-y)\overline{Q_{\om}(x-y)}dx<0\\
    b&=  -\psld{\partial_{\gamma}R(\om,\gamma,y)}{iR(\om,\gamma,y)}=\frac{-1}{\mu}\Re \int_{\R} Q_{\om}(x-y)^2dx<0\\
    c &= \frac{-1}{\mu}\norm{\partial_{k}Q_{\om}}_{L^2}^2<0.
  \end{align*}
  Implicit function theorem finally gives the existence of parameters $\tildo_{j}, \tildg_{j}, \tildx_{j}$ as functions of time. To prove that these functions are actually of class $\calC^1$, a standard regularization argument is needed. We refer to~\cite{MaMe01} for more details on that argument. 

  Now, in order to be more readable, we prove estimate~\eqref{eq:estmod} for $j=1$. In that purpose, let us write the equation of evolution satisfied  by  $\varepsilon_{1}$, namely: 
  \begin{multline}
    \label{eq:eps}
    i\partial_{t}\eps_{1}+L(\eps_{1},\eps_{2})+\calN(\eps_{1},\eps_{2})\\=-(\tildo_{1}-\frac{v_{1}^2}{4}-\partial_{t}\tildg_{1})\tildR_{1}-i\partial_{t}\tildo_{1}\partial_{\om}Q_{\tildo_{1}}e^{i(\frac{1}{2}v_{1}\cdot x+\tildg_{1})}-i(v_{1}-\partial_{t}\tildx_{1})\partial_{x} Q_{\tildo_{1}}.
  \end{multline}
  where 
  \begin{equation*}
    L(\eps_{1},\eps_{2})=\partial_{xx} \eps_{1}+\left(\mu_{1}|\tildR_{1}|^2+\beta |\tildR_{2}|^2\right)\eps_{1}+2\mu_{1}\Re(\tildR_{1}\bar{\eps_{1}})\tildR_{1}+2\beta\Re(\tildR_{2}\bar{\eps_{2}})\tildR_{1}
  \end{equation*}
  and 
  \begin{multline*}
    \calN(\eps_{1},\eps_{2})=\mu_{1}|\eps_{1}|^2\eps_{1}+\beta|\eps_{2}|^2\eps_{1}+\left(\mu_{1}|\eps_{1}|^2+\beta |\eps_{2}|^2\right)\tildR_{1}\\+\left(2\mu_{1}\Re(\tildR_{1}\eps_{1})+2\beta\Re(\tildR_{2}\bar{\eps_{2}})\right)\eps_{1}.
  \end{multline*}

  It is to be noticed that the modulation terms $\tildo_{1}-\frac{v_{1}^2}{4}-\partial_{t}\tildg_{1}$, $\partial_{t}\tildo_{1}$, and $v_{1}-\partial_{t}\tildx_{1}$ appear in the right hand side of this evolution equation. The main idea of the proof is to take in~\eqref{eq:eps} the scalar products of both sides of the equation with respectively $\tildR_{1}, i\tildR_{1}$ and $\partial_{x} \tildR_{1}$. 
  We then use  the orthogonality conditions~\eqref{eq:orth} as well as properties on $Q_{\tildo_{1}}$. In the left hand side, we transfer the derivatives acting on $\eps_{1}$ on the other side of the scalar product thanks to the modulation conditions~\eqref{eq:orth} (for the time derivatives) and integrations by parts (for the space derivatives). We finally use the equation at hand on $\tildR_{1}$: 
  \begin{multline}
    \label{eq:tildR1}
    i\partial_{t}\tildR_{1}+\partial_{xx} \tildR_{1}+\mu_{1}|\tildR_{1}|^2\tildR_{1}=\\\left(\tildo_{1}+\frac{v_{1}^2}{4}-\partial_{t}\tildg_{1}\right)\tildR_{1}+\frac{j}{\sqrt{\mu_{1}}}e^{i(\frac{1}{2}v_{1}\cdot x+\tildg_{1})}\partial_{t}\tildo_{1}\partial_{\om}Q_{\tildo_{1}}+i\frac{v_{1}-\partial_{t}\tildx_{1}}{\mu_{1}}e^{i(\frac{1}{2}v_{1}\cdot x+\tildg_{1})}\partial_{x} Q_{\tildo_{1}}.
  \end{multline} 

  For more simplicity, we only develop the computations for the scalar product with $\tildR_{1}$, the other cases are obtained using the same arguments. 

  Taking the scalar product with $\tildR_{1}$ in equation~\eqref{eq:eps} leads to:
  \begin{multline}\label{eq:spR1}
    \psld{i\partial_{t}\eps_{1}+L(\eps_{1},\eps_{2})+\calN(\eps_{1},\eps_{2})}{\tildR_{1}}=-\left(\tildo_{1}-\frac{v_{1}^2}{4}-\partial_{t}\tildg_{1}\right)\norm{\tildR_{1}}^2\\-\partial_{t}\tildo_{1}\Re \int_{\R} i e^{i(\frac{1}{2}v_{1}\cdot x+\tildg_{1})} \partial_{\om}Q_{\tildo_{1}}\overline{\tildR_{1}}dx\\-(v_{1}-\partial_{t}\tildx_{1})\Re \int_{\R} ie^{i(\frac{1}{2}v_{1}\cdot x+\tildg_{1})}\partial_{x} Q_{\tildo_{1}} \overline{\tildR_{1}}dx
  \end{multline}
  \emph{Right hand side of~\eqref{eq:spR1}:} First, using equation~\eqref{eq=Rtilde} leads to:
  \begin{equation*}
    \Re \int_{\R} i e^{i(\frac{1}{2}v_{1}\cdot x+\tildg_{1})} \partial_{\om}Q_{\tildo_{1}}\overline{\tildR_{1}}dx= \frac{1}{\sqrt{\mu_{1}}}\Im \int_{\R} \partial_{\om}Q_{\tildo_{1}}(x)Q_{\tildo_{1}}(x-\tildx_{1})dx=0
  \end{equation*}
  and 
  \begin{equation*}
    \Re \int_{\R} ie^{i(\frac{1}{2}v_{1}\cdot x+\tildg_{1})}\partial_{x} Q_{\tildo_{1}} \overline{\tildR_{1}}dx =\frac{1}{\sqrt{\mu_{1}}}\Im \int_{\R} \partial_{x} Q_{\tildo_{1}}(x)Q_{\tildo_{1}}(x-\tildx_{1}) dx=0.
  \end{equation*}  
  Thus,~\eqref{eq:spR1} reduces to:
  \begin{equation}
    \label{new44}
    \psld{i\partial_{t}\eps_{1}+L_{1}(\eps_{1},\eps_{2})+\calN(\eps_{1},\eps_{2})}{\tildR_{1}}=-\left(\tildo_{1}-\frac{v_{1}^2}{4}-\partial_{t}\tildg_{1}\right)\norm{\tildR_{1}}^2.
  \end{equation}

  \noindent
  \emph{Left hand side of~\eqref{new44}:}\\
  First, deriving modulation condition $\psld{\eps_{1}}{\tildR_{1}}=0$ with respect to time gives: $\psld{i\partial_{t}\eps_{1}}{\tildR_{1}}=\psld{\eps_{1}}{i\partial_{t}\tildR_{1}}$.

  \noindent
  Let us  now develop $\psld{L(\eps_{1},\eps_{2})}{\tildR_{1}}$: 
  \[
  \psld{L(\eps_{1},\eps_{2})}{\tildR_{1}}=\psld{\eps_{1}}{\partial_{xx} \tildR_{1}+3\mu_{1}|\tildR_{1}|^2\tildR_{1}+\beta|\tildR_{2}|^2\tildR_{1}}+2\beta\psld{\eps_{2}}{|\tildR_{1}|^2\tildR_{2}}.
  \]
  \noindent
  Finally, $\psld{\calN(\eps_{1},\eps_{2})}{\tildR_{1}}$ read:
  \begin{multline*}
    \psld{\calN(\eps_{1},\eps_{2})}{\tildR_{1}}=\psld{\mu_{1}|\eps_{1}|^2\eps_{1}+\beta|\eps_{2}|^2\eps_{1}}{\tildR_{1}}+\psld{(\mu_{1}|\eps_{1}|^2+\beta|\eps_{2}|^2)\tildR_{1}}{\tildR_{1}} \\+\psld{\left(2\mu_{1}\Re(\tildR_{1}\eps_{1})+2\beta\Re (\tildR_{2}\overline{\eps_{2}})\right)\eps_{1}}{\tildR_{1}}
  \end{multline*}

  Equation~\eqref{new44} thus leads to the left hand side term:
  \begin{multline*}
    \psld{\eps_{1}}{i\partial_{t}\tildR_{1}+\partial_{xx}\tildR_{1}+\mu_{1}|\tildR_{1}|^2\tildR_{1}}+\beta\psld{\eps_{1}}{|\tildR_{2}|^2\tildR_{1}}+2\mu_{1}\psld{\eps_{1}}{|\tildR_{1}|^2\tildR_{1}}\\+2\beta\psld{\eps_{2}}{|\tildR_{1}|^2\tildR_{2}}+\psld{\calN(\eps_{1},\eps_{2})}{\tildR_{1}}.
  \end{multline*}

  Now, using equation~\eqref{eq:tildR1} satisfied by $\tildR_{1}$ gives:
  \begin{multline*}
    \psld{\eps_{1}}{i\partial_{t}\tildR_{1}+\partial_{xx}\tildR_{1}+\mu_{1}|\tildR_{1}|^2\tildR_{1}}\\=(\tildo_{1}-\frac{v_{1}^2}{4}-\partial_{t}\tildg_{1})\psld{\eps_{1}}{\tildR_{1}}+\partial_{t}\tildo_{1}\psld{\eps_{1}}{\frac{i}{\sqrt{\mu_{1}}}e^{i(\frac{1}{2}v_{1}\cdot x+\tildg_{1})}\partial_{\om}Q_{\tildo_{1}}}\\+(v_{1}-\partial_{t}\tildx_{1})\psld{\eps_{1}}{\frac{i}{\sqrt{\mu_{1}}}e^{i(\frac{1}{2}v_{1}\cdot x+\tildg_{1})}\partial_{x} Q_{\tildo_{1}}}.
  \end{multline*}

  Finally, modulation condition $\psld{\eps_{1}}{\tildR_{1}}=0$ leads to
  \begin{multline}
    \label{eq:spR1fin}
    (\tildo_{1}-\frac{v_{1}^2}{4}-\partial_{t}\tildg_{1})\norm{\tildR_{1}}^2+\partial_{t}\tildo_{1}\psld{\eps_{1}}{\frac{i}{\sqrt{\mu_{1}}}e^{i(\frac{1}{2}v_{1}\cdot x+\tildg_{1})}\partial_{\om} Q_{\tildo_{1}}}\\+(v_{1}-\partial_{t}\tildx_{1})\psld{\eps_{1}}{{\frac{i}{\sqrt{\mu_{1}}}e^{i(\frac{1}{2}v_{1}\cdot x+\tildg_{1})}\partial_{x} Q_{\tildo_{1}}}}
    \\=-2\mu_{1}\psld{\eps_{1}}{|\tildR_{1}|^2\tildR_{1}}-\beta\psld{\eps_{1}}{|\tildR_{2}|^2\tildR_{1}}-2\beta \psld{\eps_{2}}{|\tildR_{1}|^2\tildR_{2}} -\psld{\calN(\eps_{1},\eps_{2})}{\tildR_{1}}.
  \end{multline}
  Thanks to Lemma~\ref{localisation-estimates}, it is readily seen that most terms in this equation are of order $\calO (\|\eps\|_{2})$. Finally,~\eqref{eq:spR1fin} can be re-written in a simpler way:
  \begin{equation*}
    \left(\norm{\tildR_{1}}^2+a_{1}(t)\right)(\tildo_{1}-\frac{v_{1}^2}{4}-\partial_{t}\tildg_{1})+a_{2}(t)\partial_{t}\tildo_{1}+a_{3}(t)(v_{1}-\partial_{t}\tildx_{1})=b_{1}(t)
  \end{equation*}
  where, for all $t\in [t_{0},T^n],\ \abs{a_{1}(t)}+\abs{a_{2}(t)}+\abs{a_{3}(t)}+\abs{b_{1}(t)}\leq C\norm{\eps(t)}_{2}$.

  With the same kind of arguments, taking in~\eqref{eq:eps} scalar product with $i\tildR_{1}$ and $i\partial_{x} \tildR_{1}$ respectively, we get two other equations that can be re-written as a linear system solved by the "modulation vector" 
  $\textrm{Mod}(t)=\begin{pmatrix}
    \tildo_{1}(t)-\frac{v_{1}^2}{4}-\partial_{t}\tildg_{1}(t)\\
    \partial_{t}\tildo_{1}(t)\\
    v_{1}-\partial_{t}\tildx_{1}(t)..
  \end{pmatrix}
  $ \\
  This linear system takes the following form: 
  \begin{equation*}
    \label{eq:modulationsyst}
    (\Gamma(t)+A(t))\textrm{Mod}(t)=B(t)
  \end{equation*}  
  where 
  \begin{equation*}
    \Gamma(t)=\left(\begin{array}{ccc} \norm{\tildR_{1}} & 0 & 0\\
 0 &
 \int_{\R} \partial_{\om}Q_{\tildo_{1}}(x)Q_{\tildo_{1}}(x-\tildx_{1}(t))dx
 &{0 }
 \\ 
\frac{v_{1}}{2}\norm{\tildR_{1}} & 0 & -\norm{\partial_{x}Q_{\tildo_{1}}}_{2}^2
      \end{array}\right)
  \end{equation*}
  and, with the help of Lemma~\ref{localisation-estimates}, for all $t\in [t_{0},T^n],\
  \|A(t)\|,\|B(t)\|\leq C \|\eps\|_{2}$, and hence~\eqref{eq:estmod}.
\end{proof}



\bibliographystyle{abbrv}
\bibliography{biblio}

\end{document}